\documentclass[12pt,a4paper]{amsart}

\usepackage{amssymb}
\usepackage{mathabx}
\usepackage{bbm}
\usepackage{amsthm}
\usepackage{dsfont}
\usepackage{amsmath}
\usepackage{color,graphics,srcltx}
\usepackage{bm}
\usepackage{easybmat}
\usepackage{multirow,bigdelim}
\usepackage{enumitem}
\usepackage{graphicx}
\usepackage{xcolor}

\newtheorem{proposition}{Proposition}
\newtheorem{example}[proposition]{Example}
\newtheorem{lemma}[proposition]{Lemma}
\newtheorem{corollary}[proposition]{Corollary}
\newtheorem{theorem}[proposition]{Theorem}

\newcommand{\RR}{\mathbb{R}}

\newcommand{\Cmt}{{\mathcal C}}

\definecolor{mygray}{gray}{0.6}
\newcommand{\away}{\color{mygray}}

\def\mC{\mathcal{C}}
\def\cQ{\mathcal{Q}}

\definecolor{damjana}{rgb}{.8,.2,.2}

\def\RR{\mathds R}

\def\II{\mathds I}

\def\CC{\mathcal C}

\def\FF{\mathcal F}

\begin{document}

\title[Non-exchangeability vs. concordance]{Relation between non-exchangeability and measures of concordance of copulas}

\author[D. Kokol B.]{Damjana Kokol Bukov\v{s}ek}
\address{School of Economics and Business, University of Ljubljana, and Institute of Mathematics, Physics and Mechanics, Ljubljana, Slovenia}
\email{damjana.kokol.bukovsek@ef.uni-lj.si}

\author[T. Ko\v sir]{Toma\v{z} Ko\v{s}ir}
\address{Faculty of Mathematics and Physics, University of Ljubljana, and Institute of Mathematics, Physics and Mechanics, Ljubljana, Slovenia}
\email{tomaz.kosir@fmf.uni-lj.si}

\author[B. Moj\v skerc]{Bla\v{z} Moj\v{s}kerc}
\address{School of Economics and Business, University of Ljubljana, and Institute of Mathematics, Physics and Mechanics, Ljubljana, Slovenia}
\email{blaz.mojskerc@ef.uni-lj.si}

\author[M. Omladi\v c]{Matja\v{z} Omladi\v{c}}
\address{Institute of Mathematics, Physics and Mechanics, Ljubljana, Slovenia}
\email{matjaz@omladic.net}

\begin{abstract}
An investigation is presented of how a comprehensive choice of five most important measures of concordance (namely Spearman's rho, Kendall's tau, 
Gini's gamma, Blomqvist's beta, and {their weaker counterpart Spearman's footrule}) relate to non-exchangeability, i.e., asymmetry on copulas. Besides these results, the method proposed also seems to be new and may serve as a raw model for exploration of the relationship between a specific property of a copula and some of its measures of dependence structure, or perhaps the relationship between various measures of dependence structure themselves.{} 
\end{abstract}

\thanks{All four authors acknowledge financial support from the Slovenian Research Agency (research core funding No. P1-0222).\\ \date}
\keywords{Copula; dependence concepts; 
supremum and minimum of a set of copulas; asymmetry or non-exchangeability; measures of concordance}
\subjclass[2010]{Primary:     60E05; Secondary: 60E15, 62N05}

\maketitle

\section{Introduction }

Copulas are mathematical objects that capture the dependence structure among random variables. Since they were introduced by A.\ Sklar in 1959 \cite{Skla} they have gained a lot of popularity and applications in several fields, e.g., in finance, insurance and reliability theory. Through them we model the dependence between random variables by building (bivariate) distributions with given marginal distributions. When deciding about which copulas to apply in real life scenarios the practitioners need to compare how certain statistical concepts behave on their data and on a class of copulas they intend to exploit.

An important family of the kind of notions are measures of  concordance (cf.\ \cite{Nels,DuSe}) such as Kendall's tau and Spearman's rho. Perhaps even more important is the notion of symmetry, also called exchangeability, or the lack of it. These notions have been studied extensively and increasingly. Nevertheless, the main contribution of this paper, i.e., the investigation of the relation between the two, appears to be new. What we do here is study a comprehensive choice of five most important measures of concordance {(namely Spearman's rho, Kendall's tau, Gini's gamma,  Blomqvist's beta, and their weaker counterpart Spearman's footrule)} on a narrow class of copulas that are equally nonexchangeable, in other words, they are asymmetric in an equivalent way to be defined later.

Let us present here some 
notation needed in the paper. We denote by $\CC$ the set of all bivariate copulas and by $\II$ the interval $[0,1]\subseteq\RR$. For a copula $C\in\CC$ we denote throughout the paper by $C^t$ the transpose of $C$, i.e., $C^t(u,v)=C(v,u)$ for all $(u,v)\in\II^2$. We also denote by $C^{\sigma_1}$ and $C^{\sigma_2}$ the two reflections on a copula $C$ defined by $C^{\sigma_1}(u,v) =v-C(1-u,v)$ and $C^{\sigma_2}(u,v)=u-C(u,1-v)$ (see \cite[{\S}1.7.3]{DuSe}), and by $\widehat{C}$ the survival copula of $C$. 


{Several orders can be introduced on  $\CC$. Copula $C$ is preceding $D$ in the \emph{concordance order} if $C(u,v)\leqslant D(u,v)$ and $\widehat{C}(u,v)\leqslant \widehat{D}(u,v)$ for all $(u,v)\in\II^2$ {\cite[Definition 2.4]{Joe}. Copula $C$ is preceding $D$ in the \emph{pointwise order} if only $C(u,v)\leqslant D(u,v)$ for all $(u,v)\in\II^2$. (See \cite[Definition 2.8.1]{Nels} and \cite[\S{2.11}]{Joe} for further details.) The concordance order and the pointwise order coincide on the set of two-dimensional copulas. (See \cite[\S{2.2.1}]{Joe97} for a proof of this statement.) Hence, we will simply refer to them as the order, and write $C\leqslant D$ for $C,D\in\CC$ if $C(u,v)\leqslant D(u,v)$ for all $(u,v)\in\II^2$.
It is well known that $\CC$ is 
{a partially ordered set} with respect to the order{,  but not a lattice \cite[Theorem 2.1]{NeUbF2},} and that $W(u,v)=\max\{0,u+v-1\}$ and $M(u,v)=\min\{u,v\}$ are the lower and the upper bound of all copulas, respectively. Copulas $W$ and $M$ are called the \emph{Fr\'echet-Hoeffding lower and upper bound}, respectively.

Besides these global bounds one often studies local bounds of certain sets of copulas. Perhaps among the first known examples of the kind is given in Theorem 3.2.3 of Nelsen's book \cite{Nels} (cf.\ also \cite[Theorem 1]{NQRU}), where the bounds of the set of copulas $C\in\CC$ with $C(a,b)=\theta$ for fixed $a,b\in\II$ and $\theta\in[W(a,b),M(a,b)]$ are given. In general, if $\mathcal{C}_0$ is a set of copulas, we let
\begin{equation}\label{eq:inf:sup}
  \underline{C}=\inf\mathcal{C}_0\quad\quad\overline{C} =\sup\mathcal{C}_0.
\end{equation}
In \cite{NQRU} the authors study the bounds for the set of copulas whose Kendall's tau equals a given number $t\in (-1,1)$ and for the set of copulas whose Spearman's rho equals a given number $t\in (-1,1)$. In both cases the bounds are copulas which do not belong to the set. Similar bounds for the set of copulas having a fixed value of Blomqvist's beta were found in \cite{NeUbF}. In \cite{BBMNU14} the authors present the local bounds for the set of copulas having a fixed value of the degree of nonexchangeability. Observe that this amounts to the same as studying the set of copulas having a fixed value of the measure of asymmetry $\mu_{\infty}$ defined in Section 3 of this paper since the two measures of nonexchangeability are a multiple of each other {(by a factor of $3$)}. However, this set might be a bit too big if we want to see how the value of the measure of asymmetry relates to the value of a measure of concordance.

We are now in position to explain some of our main results. 
The authors of \cite{KoBuKoMoOm} introduce in \S2 following the ideas of \cite{KlMe} the so-called \emph{maximal asymmetry function}. This function determines maximal possible asymmetry $d^*(a,b)$ on the set of all copulas $\mathcal{C}$ at a given point $(a,b)$ of the unit square $\II^2$. Given a fixed $c\in[0,d^*(a,b)]$ we define $\mathcal{C}_0$ as the set of copulas for which the value of the maximal asymmetry function at $(a,b)$ equals $c$. We determine the local bounds of this set given by \eqref{eq:inf:sup}. For any $c$ of the kind we compute the possible range of a certain measure of concordance on this set. We believe that the so defined set is more appropriate to investigate the relationship between asymmetry and measures of concordance than the set proposed in \cite{BBMNU14}.

The paper is organized as follows. Section 2 contains preliminaries on 
asymmetry leading to Theorem \ref{th1}, one of our main results. 
{Section 3 contains preliminaries on measures of concordance. }  In Section 4 we present the elaboration of measures of concordance on 
{local bounds described in Theorem \ref{th1}}. 
This part of the paper is technically quite involved. The relations between the two notions are then studied in Sections 5 to 9 for Spearman's rho, Kendall's tau, Spearman's footrule, Gini's gamma, and Blomqvist's beta, respectively. So, among the main results of the paper we should point out Theorems \ref{th:rho}, \ref{th:tau}, \ref{th:phi}, \ref{th:gamma} and \ref{th:beta}, and Corollaries \ref{cor:rho}, \ref{cor:tau}, \ref{cor:phi}, \ref{cor:gamma} and \ref{cor:beta}, together with Figures \ref{fig-rho} to \ref{fig-beta}. Some concluding thoughts are given in Section 10.

\section{Maximally asymmetric copulas }\label{sec:max asym}

In this paper we study relations between various measures of concordance and a measure of non-exchangeability $\mu_\infty$. The topic of non-exchangeability or asymmetry of copulas is attracting much attention (see e.g. \cite{DeBDeMJw,DM,DP2,DP1,FeSaUbFl,GeNe,KlMe,KoBuKoMoOm,N}).
To quantify the non-exchangeability or asymmetry the authors in \cite{DKSU-F} introduced the notion of a measure of asymmetry. {We use their axioms in the following definition. Other sets of axioms might be also reasonable from applications point of view.}

A function $\mu: \mC\to [0,\infty)$ is {\em a measure of asymmetry (or a measure of non-exchangeability) for copula $C$} if it
satisfies the following properties:
\begin{description}
\item[(B1)] there exists $K\in [0,\infty)$ such that, for all $C\in\mC$ we have $\mu(C) \leqslant K$,
\item[(B2)] $\mu(C) = 0$ if and only if $C$ is symmetric, i.e. $C=C^t$,
\item[(B3)] $\mu(C) = \mu(C^t)$ for every $C\in\mC$,
\item[(B4)] $\mu(C) = \mu(\widehat{C})$ for every $C\in\mC$,
\item[(B5)] if $(C_n)_{n\in\mathbb{N}}$ and $C$ are in $\mC$, and if $(C_n)_{n\in\mathbb{N}}$ converges uniformly to $C$, then $(\mu(C_n))_{n\in\mathbb{N}}$ converges to $\mu(C)$.
\end{description}

A large class of measures of asymmetry is provided in \cite[Theorem 1]{DKSU-F}: Let $d_p$ be the classical $L_p$ distance in $\mC$ for $p \in [1,\infty]$. When $p =\infty$, we have
$$d_{\infty}(C, D) = \max_{u,v\in \II}\left|C(u,v)-D(u,v)\right|.$$
Then, the measure of asymmetry $\mu_{\infty}: \mC\to [0,\infty)$ is given by
$$\mu_{\infty}(C) = d_{\infty}(C,C^t ).$$
{Sometimes in the literature this measure is normalized to $3\mu_{\infty}$ (e.g. \cite{BBMNU14}).}

In \cite[\S 2]{KoBuKoMoOm}, the so-called \emph{maximal asymmetry function} is introduced; its value at a fixed point $(u, v)\in\II^2$ is given by
\[
    d^*_\FF (u,v) = \sup_{C\in\FF} \{|C(u,v)-C(v,u)|\},
\]
where $\FF\subseteq\CC$ is an arbitrary family of copulas. If $\FF=\CC$ this supremum is attained since $\CC$ is a compact set by \cite[Theorem 1.7.7]{DuSe}. Klement and Mesiar \cite{KlMe} 
showed that
\begin{equation}\label{eq:kle mes}
  d_\CC^*(u,v)=\min\{u,v,1-u,1-v,|v-u|\}.
\end{equation}
{Confer also \cite{N} where only $\leqslant$ inequality is established.}

Now, choose $(a,b)\in\II^2$.  Furthermore, choose a $c\in\II$ such that
\[
    0\leqslant c\leqslant d_\CC^*(a,b).
\]

According to \cite[Theorem 1.4.5]{DuSe} the set $\CC$ is convex, so it follows easily that there exits a copula $C\in\CC$ such that
\begin{equation}\label{asymmetry point}
  C(a,b)-C(b,a) = c.
\end{equation}
{
\begin{lemma}\label{lem1}
For any $c$ such that $0\leqslant c\leqslant\min\{a,b,1-a,1-b\}$ functions
  \begin{equation}\label{max asym copula left}
    \underline{C}^{(a,b)}_c(u,v)=\max\{W(u,v),\min\{d_1,u-a+d_1,v-b+d_1,u+v-a-b+d_1\}\},
  \end{equation}
  and
  \begin{equation}\label{max asym copula right}
    \overline{C}^{(a,b)}_{c}(u,v)=\min\{M(u,v),\max\{d_2,u-b+d_2,v-a+d_2,u+v-a-b+d_2\}\},
  \end{equation}
  are copulas. Here $W$ respectively $M$ are the Fr\'{e}chet Hoeffding lower bound respectively upper bound,
  \begin{equation}\label{d_1}
  d_1=d_{1,c}^{(a,b)}=W(a,b)+c
  \end{equation}
  and
  \begin{equation}\label{d_2}
  d_2=d_{2,c}^{(a,b)}=M(a,b)-c.
  \end{equation}
Furthermore, the following relations hold:
\begin{equation}\label{rel Csymm}
\overline{C}^{(a,b)}_{c}=(\overline{C}^{(b,a)}_{c})^t,\quad \underline{C}^{(a,b)}_{c}=(\underline{C}^{(b,a)}_{c})^t,
\end{equation}
\begin{equation}\label{rel Cabc}
\overline{C}^{(a,b)}_{c}=(\underline{C}^{(b,1-a)}_c)^{\sigma_1}=(\underline{C}^{(1-b,a)}_c)^{\sigma_2},
\end{equation}
and
\begin{equation}\label{rel d12}
d_{2,c}^{(a,b)}=a-d_{1,c}^{(1-b,a)}=b-d_{1,c}^{(b,1-a)}.
\end{equation}
\end{lemma}

\begin{figure}[h]
            \includegraphics[width=5cm]{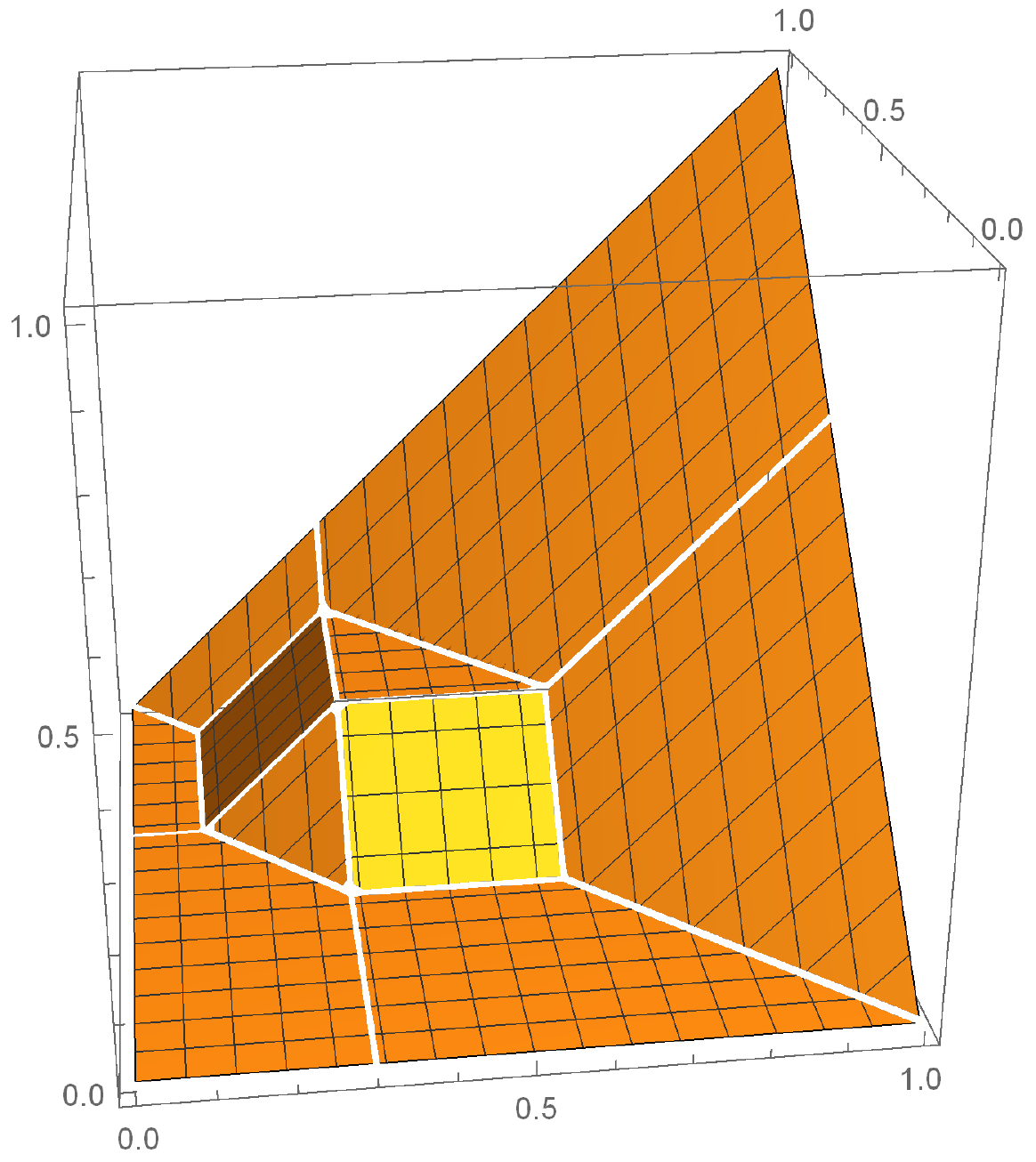} \hfil \includegraphics[width=5cm]{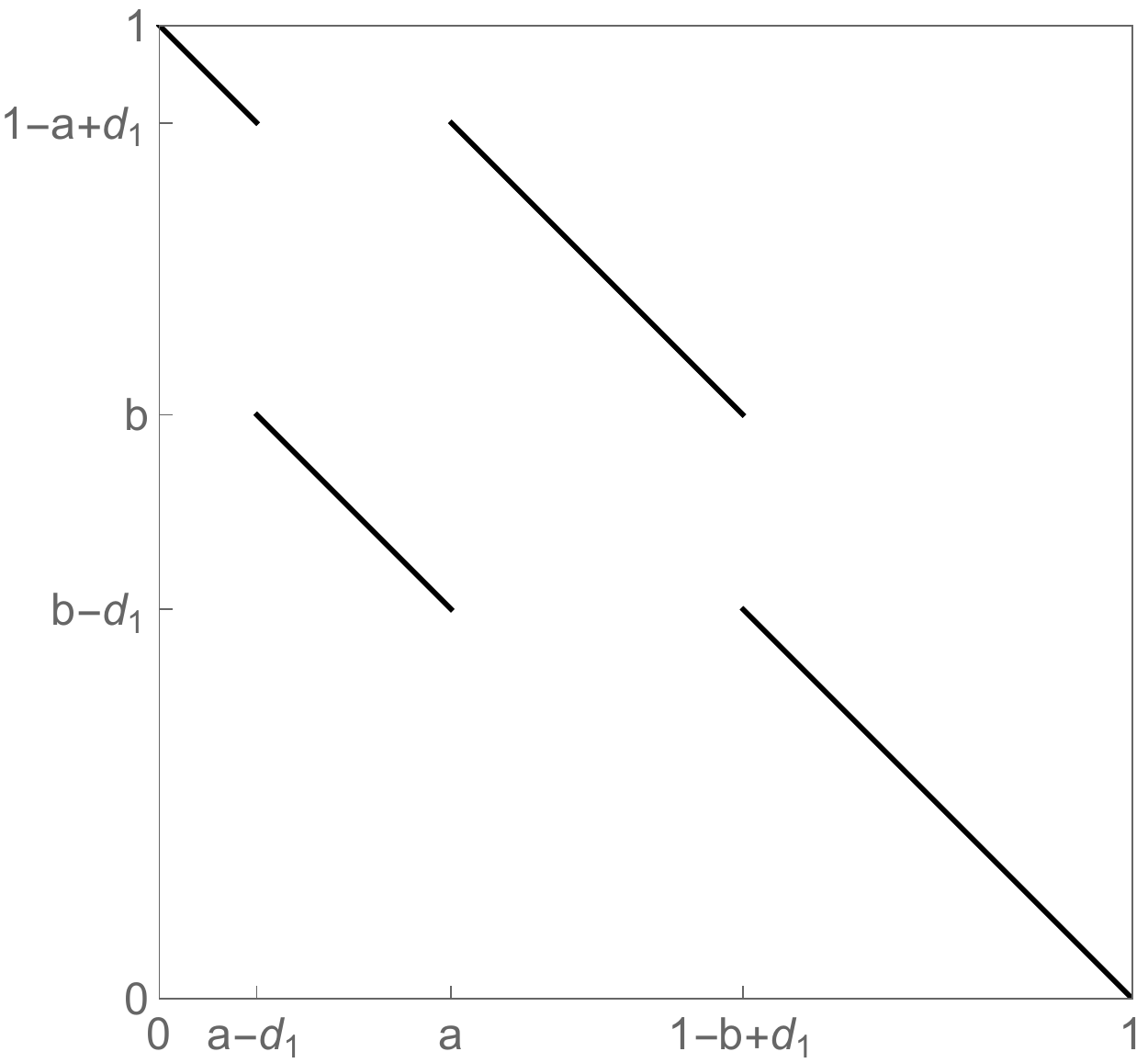}
            \caption{The surface of $\underline{C}^{(a,b)}_{c}$ and its scatterplot }\label{fig1}
\end{figure}

\begin{figure}[h]
            \includegraphics[width=5cm]{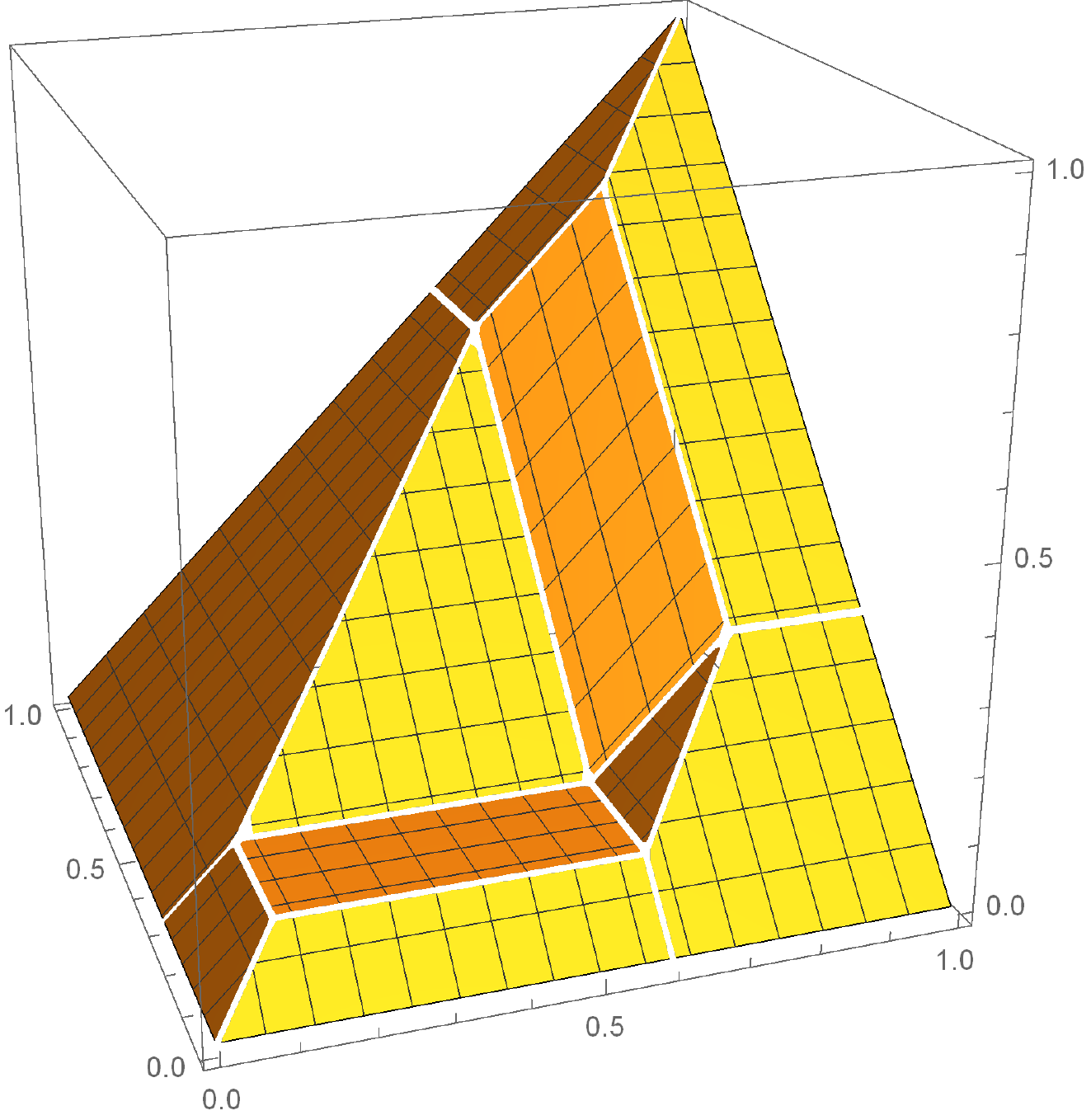} \hfil \includegraphics[width=5cm]{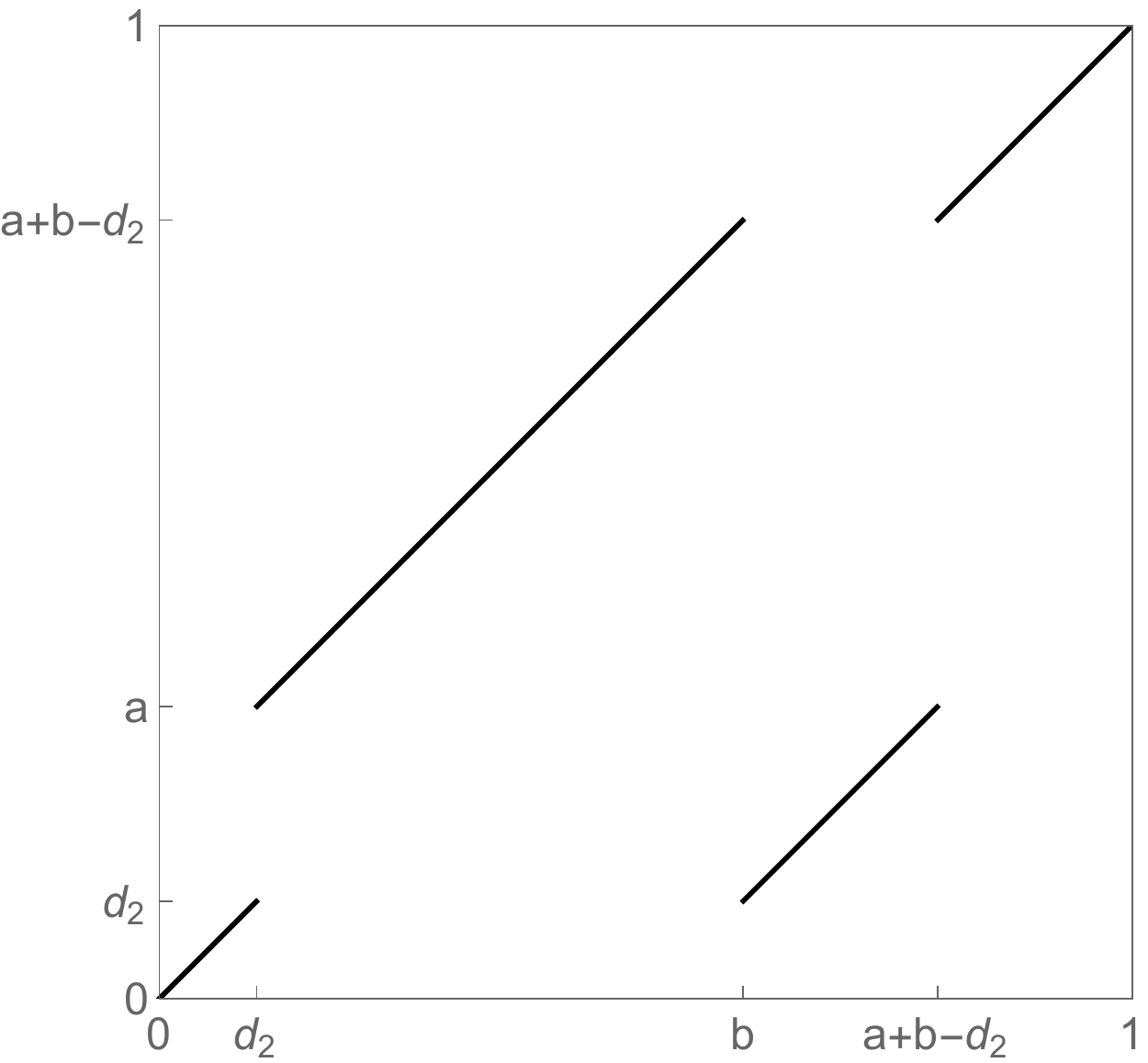}
            \caption{The surface of $\overline{C}^{(a,b)}_{c}$ and its scatterplot }\label{fig2}
\end{figure}

\begin{proof}
Functions $\underline{C}^{(a,b)}_{c}(u,v)$ and $\overline{C}^{(a,b)}_{c}(u,v)$ are exactly copulas $C_L$ and $C_U$, respectively, given in \cite[Theorem 3.2.3]{Nels}. Observe that reflecting the scatterplot of $\overline{C}^{(a,b)}_{c}(u,v)$ (see Figure \ref{fig2}) over the lines $x=\frac12$ and $y=\frac12$, respectively, gives the scatterplots of $\underline{C}^{(1-a,b)}_{c}$ and $\underline{C}^{(a,1-b)}_{c}$, respectively (see Figure \ref{fig1}). Thus equalities in \eqref{rel Cabc} hold. A simple comparison of parameters also shows that the equalities in \eqref{rel d12} hold. Finally, equalities in \eqref{rel Csymm} follow if we reflect the two scatterplots in Figures \ref{fig1} and \ref{fig2} along the main diagonal of the unit square.
\end{proof}

Observe that $c$ is small enough so that everywhere close to the boundary of the square $\II^2$ copula $W$ prevails in the definition of $\underline{C}^{(a,b)}_{c}$ and copula $M$ prevails in the definition of $\overline{C}^{(a,b)}_{c}$. Note that $\underline{C}^{(a,b)}_{c}$ is a {flipped} shuffle of 
{$M$ (see \cite[{\S}3.6]{DuSe})} 
and $\overline{C}^{(a,b)}_{c}$ is a (straight) shuffle of $M$. Observe also that condition $c\in[0,d_\CC^*(a,b)]$ is stronger that the assumption $0\leqslant c\leqslant\min\{a,b,1-a,1-b\}$ on $c$ in Lemma \ref{lem1}. Then, equality \eqref{asymmetry point} fails for $c$ such that $0\leqslant c\leqslant\min\{a,b,1-a,1-b\}$ but $c>|b-a|$.
}

\begin{theorem} \label{th1}
  For any $c\in[0,d_\CC^*(a,b)]$ there exist copulas $\underline{C}, \overline{C}$ satisfying Condition \eqref{asymmetry point} such that for every copula $C$ satisfying \eqref{asymmetry point} we have $\underline{C}\leqslant C\leqslant \overline{C}$.
\end{theorem}

\begin{figure}[h]
            \includegraphics[width=5cm]{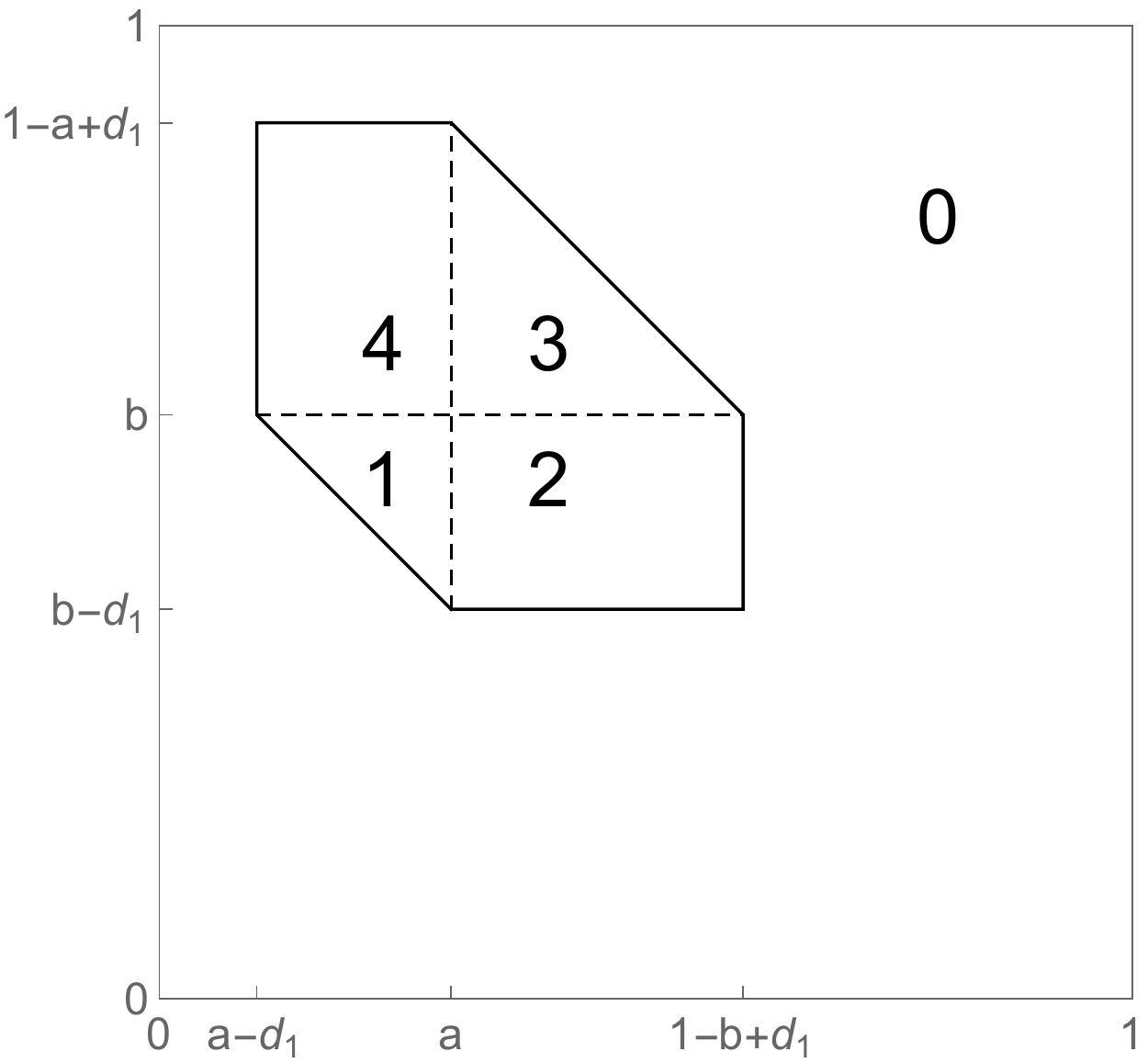} \hfil \includegraphics[width=5cm]{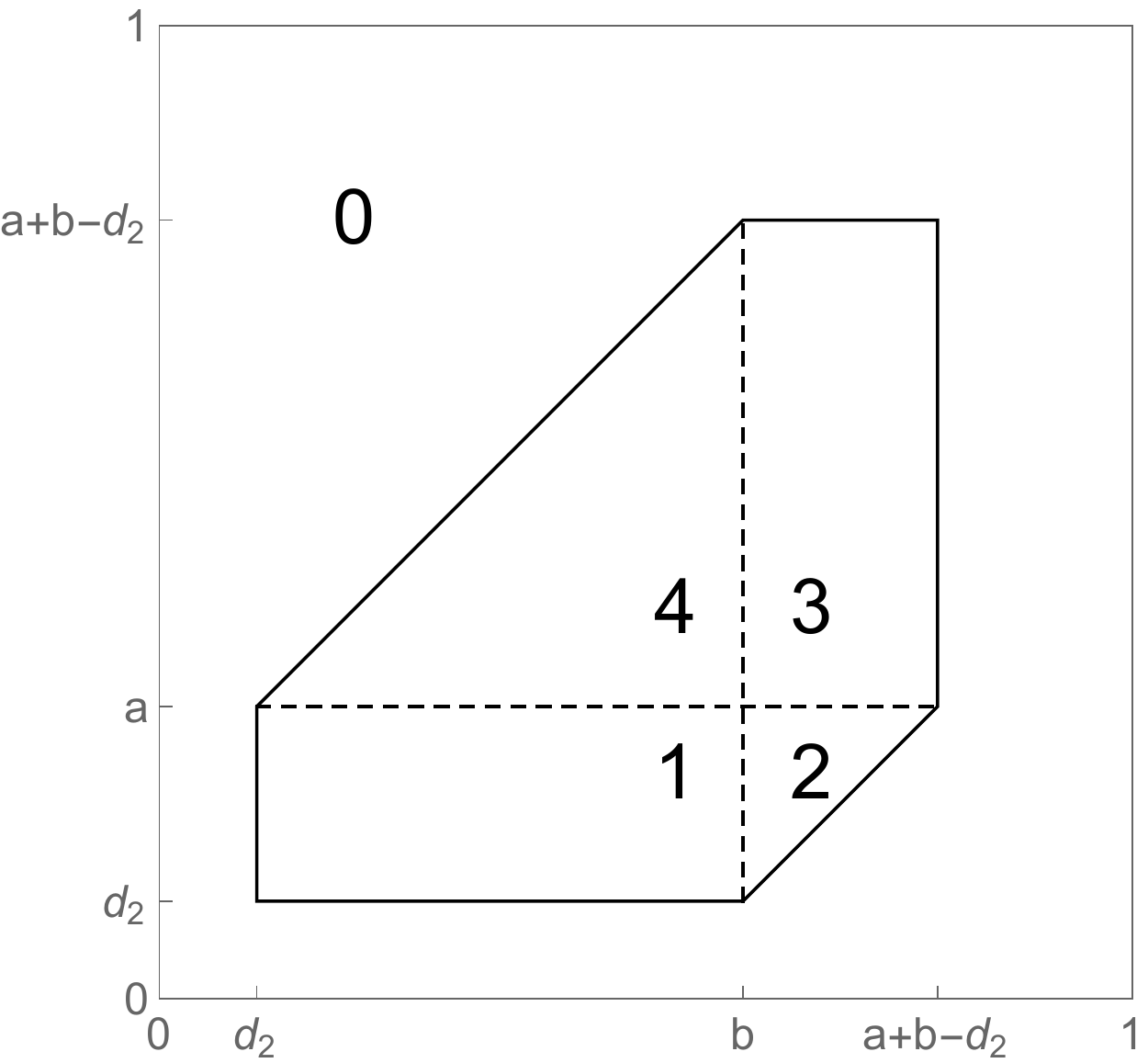}
            \caption{The five cases -- sections -- considered in the proof for copula $\underline{C}$, respectively $\overline{C}$. }\label{podrocja}
\end{figure}

\begin{proof}
{We use the copulas introduced in Lemma \ref{lem1}. We write $\underline{C}(u,v)=\underline{C}^{(a,b)}_{c}(u,v)$ and $\overline{C}(u,v)=\overline{C}^{(a,b)}_{c}(u,v)$. Note that $d_1$ and $d_2$ are given in \eqref{d_1} and \eqref{d_2}, respectively.} 
{It} is clear that $\underline{C}$ satisfies Condition \eqref{asymmetry point}, since $\underline{C}(b,a)=W(b,a)$ and $\underline{C}(a,b)=d_1=W(a,b)+c$. The fact that $\overline{C}$ satisfies this condition goes similarly using the definition of $d_2$.  \\

  {Suppose that $C$ is a copula satisfying \eqref{asymmetry point}.} Let us now show that $\underline{C}\leqslant C$. It is clear that $C(b,a)\geqslant W(b,a)=\underline{C}(b,a)$ and
  \begin{equation}\label{star star}
    C(a,b)=C(b,a)+c\geqslant W(b,a)+c=W(a,b)+c=\underline{C}(a,b)=d_1.
  \end{equation}
  The general proof will be given in 5 cases, corresponding to the areas, shown in Figure \ref{podrocja}.

    Case 0: If $(u,v)$ does not belong to any of the sets 1, 2, 3, or 4, then $\underline{C}=W$ and we are done.

    Case 1:  $a\geqslant u, b\geqslant v, a+b-d_1\leqslant u+v$. Here $\underline{C}(u,v)=u+v-a-b+d_1$.
    Since
    \[
     C(a,v)-C(u,v)\leqslant a-u
    \]
    by the 1-Lipschitz property, and
  \[
     C(a,b)-C(a,v)\leqslant b-v
    \]
  by the same property, it follows that
  \[
    C(u,v)\geqslant C(a,v)-(a-u)\geqslant C(a,b)-(a-u)-(b-v) \geqslant
    \]
    \[
        \geqslant d_1-(a-u)-(b-v) = \underline{C}(u,v),
  \]
    where we have also used \eqref{star star}.

    Case 2: $a\leqslant u\leqslant 1-b+d_1, b-d_1\leqslant v\leqslant b$. Here $\underline{C}=v-b+d_1$, so that
  \[
    C(u,v)\geqslant C(a,v)\geqslant C(a,b)-(b-v)\geqslant \underline{C}(u,v),
  \]
  where we have used respectively the 1-increasing property, the 1-Lipschitz property, and \eqref{star star}.

    Case 3: $a\leqslant u, b\leqslant v, a+b-d_1\geqslant u+v$. Here $\underline{C}=d_1$, so that
  \[
    C(u,v)\geqslant C(a,b)\geqslant d_1=\underline{C}(u,v),
  \]
  where in the first inequality we used twice the 1-increasing property, and in the second one we used \eqref{star star}.

 Case 4: $a-d_1\leqslant u \leqslant a, b\leqslant v\leqslant 1-a+d_1$. Here $\underline{C}=u-a+d_1$, so that
  \[
    C(u,v) \geqslant C(u,b) \geqslant C(a,b)-(a-u) \geqslant \underline{C}(u,v).
  \]
  In the displayed considerations we used the 1-increasing property, the 1-Lipschitz property and then \eqref{star star}.

  Finally, we show that $C\leqslant \overline{C}$. We first observe that $\overline{C}(b,a)=d_2$ and $\overline{C}(a,b)=M(a,b)$ so that the desired property is valid for the two ``central'' points; and then, as above, we use the 5 cases of sections where $\overline{C}$ is either equal to copula $M$ or a simple plane in order to get the property in general.
\end{proof}

Observe that the two copulas 
$\underline{C}^{(a,b)}_{c}$ 
and 
 $\overline{C}^{(a,b)}_{c}$ 
  are the exact lower, respectively upper, bound of all asymmetric copulas that satisfy \eqref{asymmetry point}. We call these copulas \emph{local Fr\'echet-Hoeffding bounds}.

\begin{example}\label{ex1}
 Observe that for $a=0.4$, $b=0.6$ and $c=0.1$ we have that $\underline{C}^{(a,b)}_{c}\leqslant\Pi\leqslant \overline{C}^{(a,b)}_{c}$.  Hence the set of all copulas $C$ such that $\underline{C}^{(a,b)}_{c}\leqslant C\leqslant \overline{C}^{(a,b)}_{c}$ is not in general the same as the set of all copulas $C$ such that $C(a,b)-C(b,a)=c$.
\end{example}

\section{Preliminaries on measures of concordance}

A pair of random vectors $(X_1,X_2)$ and $(Y_1,Y_2)$ is \emph{concordant} if larger values of the first one are associated with larger values of the second one, while smaller values of the first one are associated with smaller values of the second. The opposite notion is the notion of discordance. A pair of random vectors is \emph{discordant} if larger values for the first one are associated with smaller values of the second one, while smaller values of the first one are associated with larger values of the second. With this in mind, we denote by $\cQ$ (see \cite[\S 5.1]{Nels} or \cite[\S 2.4]{DuSe}) the difference of two probabilities $\cQ=P((X_1-X_2)(Y_1-Y_2)>0)-P((X_1-X_2)(Y_1-Y_2)<0)$ for a pair of { continuous} random vectors $(X_1,X_2)$ and $(Y_1,Y_2)$. If the corresponding copulas are $C_1$ and $C_2$ then we have
\begin{equation}\label{concordance}
  \cQ=\cQ(C_1,C_2)= 4 \int_{\II^2} C_2(u,v)dC_1(u,v) -1.
\end{equation}
See \cite[Theorem 5.1.1]{Nels}. Function $\cQ$ is called the \emph{concordance function}. It was introduced by Kruskal \cite{Krus}. This function has a number of useful properties \cite[Corollary 5.1.2]{Nels}:
\begin{description}
  \item[{(Q1)}] It is symmetric in the two arguments: {$\cQ(C_1,C_2)=\cQ(C_2,C_1)$.}
 \item[{(Q2)}] It is nondecreasing in each argument: {If $C_1\leqslant C^{\prime}_1$ and $C_2\leqslant C^{\prime}_2$ then\\ $\cQ(C_1,C_2)\leqslant\cQ(C^{\prime}_1,C_2)$ and $\cQ(C_1,C_2)\leqslant\cQ(C_1,C^{\prime}_2)$.}
 \item[{(Q3)}] It remains unchanged when both copulas are replaced by their survival copulas: \\ {$\cQ(C_1,C_2)=\cQ(\widehat{C}_2,\widehat{C}_1)$.}
\end{description}

{Next, we state an additional property of $\cQ$ that is probably well-known but we could not find it stated in the form we use later. It is known that the dihedral group $D_4$ acts on $\mC$ via transposition and reflections (see, e.g., \cite[Table 2.1]{EdMiTa}). The \emph{direct orbit} of this action on a given copula $C\in\mC$ is $\{C,\widehat{C},C^t,\widehat{C}^t\}$ and its \emph{opposite orbit} is $\{C^{\sigma_1},C^{\sigma_2},(C^t)^{\sigma_1},(C^t)^{\sigma_2}\}$. The concordance function is invariant under the direct action on both arguments, while the opposite action changes the sign. The former follows by (Q2) and simple observation that $\cQ(C_1,C_2)=\cQ(C_1^t,C_2^t)$ holds since the value of the integral in \eqref{concordance} remains the same if we exchange $u$ and $v$. To see the latter we recall \cite[Lemma 2.3]{EdMiTa}, where it is proved that
\begin{equation}\label{sigmaQ}
\int_{\II^2} C_2(u,v)dC_1^{\sigma_j}(u,v)+\int_{\II^2} C_2^{\sigma_j}(u,v)dC_1(u,v)=\frac12\textrm{ for } j=1,2.
\end{equation}
A straightforward consequence of \eqref{sigmaQ} is the following property of the concordance function:
\begin{description}
  \item[{(Q4)}] When both copulas are replaced by their reflected copulas the sign changes: \\ $\cQ(C_1^{\sigma_j},C_2^{\sigma_j})=-\cQ(C_1,C_2)$ for $j=1,2$.
\end{description}
Note that since reflection is an involution Property (Q4) can be restated also in the form $\cQ(C_1^{\sigma_j},C_2)=-\cQ(C_1,C_2^{\sigma_j})$ for $j=1,2$.}

A mapping $\kappa:\CC\to [-1,1]$ is called a \emph{measure of concordance} if it satisfies the following properties (see 
{\cite[Definition 5.1.7]{Nels}}):
\begin{description}
{\item[(C1)] $\kappa(M)=1$.}
\item[(C2)] $\kappa(C)=\kappa(C^t)$ for every $C\in\CC$.
\item[(C3)] $\kappa(C)\leqslant\kappa(D)$ when $C\leqslant D$.  \label{monotone}
\item[(C4)] $\kappa(C^{\sigma_1})=\kappa(C^{\sigma_2})=-\kappa(C)$. \label{kappa_flipped}
\item[(C5)] If a sequence of copulas $C_n$, $n\in\mathbb{N}$, converges uniformly to $C\in\CC$, then $\lim_{n\to\infty}\kappa(C_n)=\kappa(C)$.
\end{description}
We will refer to property (\ref{monotone}) above simply by saying that a measure of concordance under consideration is \emph{monotone}.

{Observe that some more properties that are sometimes stated in the definition follow from the properties listed above. Namely, we also have that a measure of concordance satisfies the following:
\begin{description}
\item[(C6)] $\kappa(\Pi)=0$, where $\Pi$ is the independence copula $\Pi(u,v)=uv$. \label{kappa_Pi}
\item[(C7)] $\kappa(W)=-1$.\label{kappaW}
\item[(C8)] $\kappa(C)=\kappa(\widehat{C})$ for every $C\in\CC$.
\end{description}
Property (C4) implies Property (C6) since $\Pi$ is invariant under both reflections, Property (C7) since $W=(M)^{\sigma_1}$, and Property (C8) since $\widehat{C}=\left(C^{\sigma_1}\right)^{\sigma_2}$.}

Note that properties (C2), (C4) and (C8) can be expressed also in terms of the $D_4$ action on $\CC$. Thus, a measure of concordance leaves a direct orbit of a copula $C\in\CC$ invariant, while it changes sign on the opposite orbit.
\\

{Because of their significance in statistical analysis measures of concordance and their relatives measures of association and measures of dependence are a classical topic of research. It was Scarsini \cite{Scar} who introduced formal axioms of a measure of concordance. Some of more recent references on bivariate measures of concordance include \cite{EdTa,FrNe,FuSch,Lieb,NQRU,NQRU2}. Their multivariate generalization were studied e.g. in \cite{BeDoUF2,DuFu,Tayl,UbFl}. For bivariate copulas the main measures of concordance are naturally studied through symmetries that are consequence of properties (Q1), (Q3) and (Q4) of the concordance function $\cQ$ (see for instance \cite{BeDoUF,EdMiTa,EdMiTa2}). }

{The 
{four} most commonly used measures of concordance of a copula $C$ are Kendall's tau, Spearman's rho, 
Gini's gamma and Blomqvist's beta. {We refer to \cite{Lieb} for an extended definition of a measure of concordance. Roughly, if we replace Property (C4) by Property (C6) we get what Liebscher in \cite{Lieb} calls a \emph{weak measure of concordance}. Spearman's footrule is an example of such a weak measure of concordance.} To simplify the statements from now on, we refer simply as measures of concordance to all five of measures studied: Kendall's tau, Spearman's rho, Gini's gamma, Blomqvist's beta, and Spearman's footrule. Statistical significance of all five measures is already well established. See e.g. \cite[Ch. 5]{Nels} for Kendall's tau, Spearman's rho, and e.g. \cite{GeNeGh,UbFl} for Gini's gamma, Blomqvist's beta and Spearman's footrule.}

The first  four of the measures of concordance may be defined in terms of the concordance function $\cQ$. \emph{Kendall's tau} of $C$ is defined by
\begin{equation}\label{tau}
\tau(C)=\cQ(C,C),
\end{equation}
\emph{Spearman's rho} by
\begin{equation}\label{rho}
\rho(C)=3\cQ(C,\Pi),
\end{equation}
\emph{Gini's gamma} by
\begin{equation}\label{gamma}
\gamma(C)=\cQ(C,M)+\cQ(C,W).
\end{equation}
\emph{Spearman's footrule} by
\begin{equation}
\phi(C)=\textstyle\frac12\left(3\cQ(C,M)-1\right).
\end{equation}
On the other hand, \emph{Blomqvist's beta} is defined by
\begin{equation}\label{beta}
\beta(C)=\textstyle4C\left(\frac12,\frac12\right)-1.
\end{equation}
See \cite[{\S}2.4]{DuSe} and \cite[Ch. 5]{Nels} for further details.

\section{Concordance function evaluated at the local Fr\'echet-Hoeffding bounds} 

In what follows we will need the values of the concordance function $\cQ$ applied to a pair of copulas one of which will be lower bound $\underline{C}^{(a,b)}_{c}$ respectively upper bound $\overline{C}^{(a,b)}_{c}$ defined by \eqref{max asym copula left} respectively by \eqref{max asym copula right}. Observe that these copulas are {flipped and straight} shuffles of 
$M$ {\cite[p. 103]{DuSe} respectively (see Figures \ref{fig1} and \ref{fig2}). A short computation then reveals that
\begin{equation}\label{concordance_left}
\begin{split}
   \cQ(D,\underline{C}^{(a,b)}_{c}) &= 4\int D(u,v)d\underline{C}^{(a,b)}_{c}(u,v)-1 = \\
                      &= 4\int_{0}^{a-d_1}D(u,1-u)du + 4\int_{a-d_1}^{a}D(u,a+b-d_1-u)du +\\
                    & \phantom{x} + 4\int_{a}^{1-b+d_1}D(u,1+d_1-u)du + 4\int_{1-b+d_1}^{1}D(u,1-u)du -1
\end{split}
\end{equation}
respectively
\begin{equation}\label{concordance_right}
 \begin{split}
   \cQ(D,\overline{C}^{(a,b)}_{c}) &= 4\int D(u,v)d\overline{C}^{(a,b)}_{c}(u,v)-1 = \\
                      &= 4\int_{0}^{d_2}D(u,u)du + 4\int_{d_2}^{b}D(u,u+a-d_2)du + \\
                                        & \phantom{x} + 4\int_{b}^{a+b-d_2}D(u,u-b+d_2)du + 4\int_{a+b-d_2}^{1}D(u,u)du -1.
\end{split}
\end{equation}

To compute the values of various measures of concordance we need the above values of $\cQ$ for various copulas $D$: $W$, $\Pi$, $M$, and $\underline{C}^{(a,b)}_{c}$, respectively
$\overline{C}^{(a,b)}_{c}$. Recall that $d_1=d_{1,c}^{(a,b)}$ is given by (\ref{d_1}).

\begin{proposition} \label{prop2}
Let $(a,b)\in\II^2$ and $0\leqslant c\leqslant\min\{a,b,1-a,1-b\}$. For copulas $\underline{C}^{(a,b)}_{c}$ it holds:
\begin{enumerate}
\item $\cQ(W, \underline{C}^{(a,b)}_{c}) = 4d_1(1-a-b+d_1) -1,$
\item $\cQ(\Pi, \underline{C}^{(a,b)}_{c}) = 2d_1(1-a-b+d_1)(1-a-b+2d_1)-\frac13,$
\item $\cQ(M, \underline{C}^{(a,b)}_{c}) =$ \\  $=\left\{ \begin{array}{ll}
        0;                        & \text{if } b \geqslant d_1 + \frac12, \vspace{1mm}\\
        (2d_1+1-2b)^2;            & \text{if } \frac12(1+d_1) \leqslant b \leqslant d_1 + \frac12, a \leqslant b-d_1, \vspace{1mm}\\
        (a+b-1-d_1)(3b-3d_1-a-1); & \text{if } \frac12(1+d_1) \leqslant b \leqslant d_1 + \frac12, a \geqslant b-d_1, \vspace{1mm}\\
        d_1(2+3d_1-4b);           & \text{if } b \leqslant \frac12(1+d_1), a \leqslant b-d_1, \vspace{1mm}\\
 { 2d_1(1+d_1-a-b)-(a-b)^2; } & { \text{if } d_1 \geqslant 2a-1, d_1 \geqslant 2b-1, d_1 \geqslant a-b,d_1 \geqslant b-a,} \vspace{1mm}\\
        d_1(2+3d_1-4a);           & \text{if } a \leqslant \frac12(1+d_1), b \leqslant a-d_1, \vspace{1mm}\\
        (a+b-1-d_1)(3a-3d_1-b-1); & \text{if } \frac12(1+d_1) \leqslant a \leqslant d_1 + \frac12, b \geqslant a-d_1, \vspace{1mm}\\
        (2d_1+1-2a)^2;            & \text{if } \frac12(1+d_1) \leqslant a \leqslant d_1 + \frac12, b \leqslant a-d_1, \vspace{1mm}\\
                0;                        & \text{if } a \geqslant d_1 + \frac12,
                \end{array} \right.$
\item $\cQ(\underline{C}^{(a,b)}_{c}, \underline{C}^{(a,b)}_{c}) = 4d_1(1-a-b+d_1) -1.$
\end{enumerate}
\end{proposition}

\begin{proof}
\begin{enumerate}
    \item Observe that $W(u,1-u)=W(u,a+b-d_1-u)=0$, so the first, the second and the fourth integral in \eqref{concordance_left} are zero. Further, $W(u,1+d_1-u)=d_1$, hence $\cQ(W,\underline{C}^{(a,b)}_{c})=4d_1(1-a-b+d_1) -1$.
    \item Observe that $\Pi(u,a+b-d_1-u)=u(a+b-d_1-u)=u(1-u)+u(a+b-d_1-1)$ and $\Pi(u,1+d_1-u)=u(1+d_1-u)=u(1-u)+u d_1$. It follows that
    \[\begin{split}
    \cQ(\Pi,\underline{C}^{(a,b)}_{c})  &= 4 \int_0^1 u(1-u) du + 4\int_{a-d_1}^a u(a+b-d_1-1) du + 4\int_a^{1-b+d_1} u d_1 du -1 =\\ &=  2d_1(1-a-b+d_1)(1-a-b+2d_1)-\frac13.
    \end{split}\]
    \item The value of $\cQ(M, \underline{C}^{(a,b)}_{c})$ depends on where the vertices $P(a, b-d_1)$, $Q(1-b+d_1, b-d_1)$,
    $R(1-b+d_1, b)$, $S(a, 1-a+d_1)$, $T(a-d_1, 1-a+d_1)$, and $U(a-d_1, b)$ of the hexagon $PQRSTU$ lie with respect to the main diagonal $u=v$
    (see Figure \ref{sestkotnik}).

\begin{figure}[h]
            \includegraphics[width=5cm]{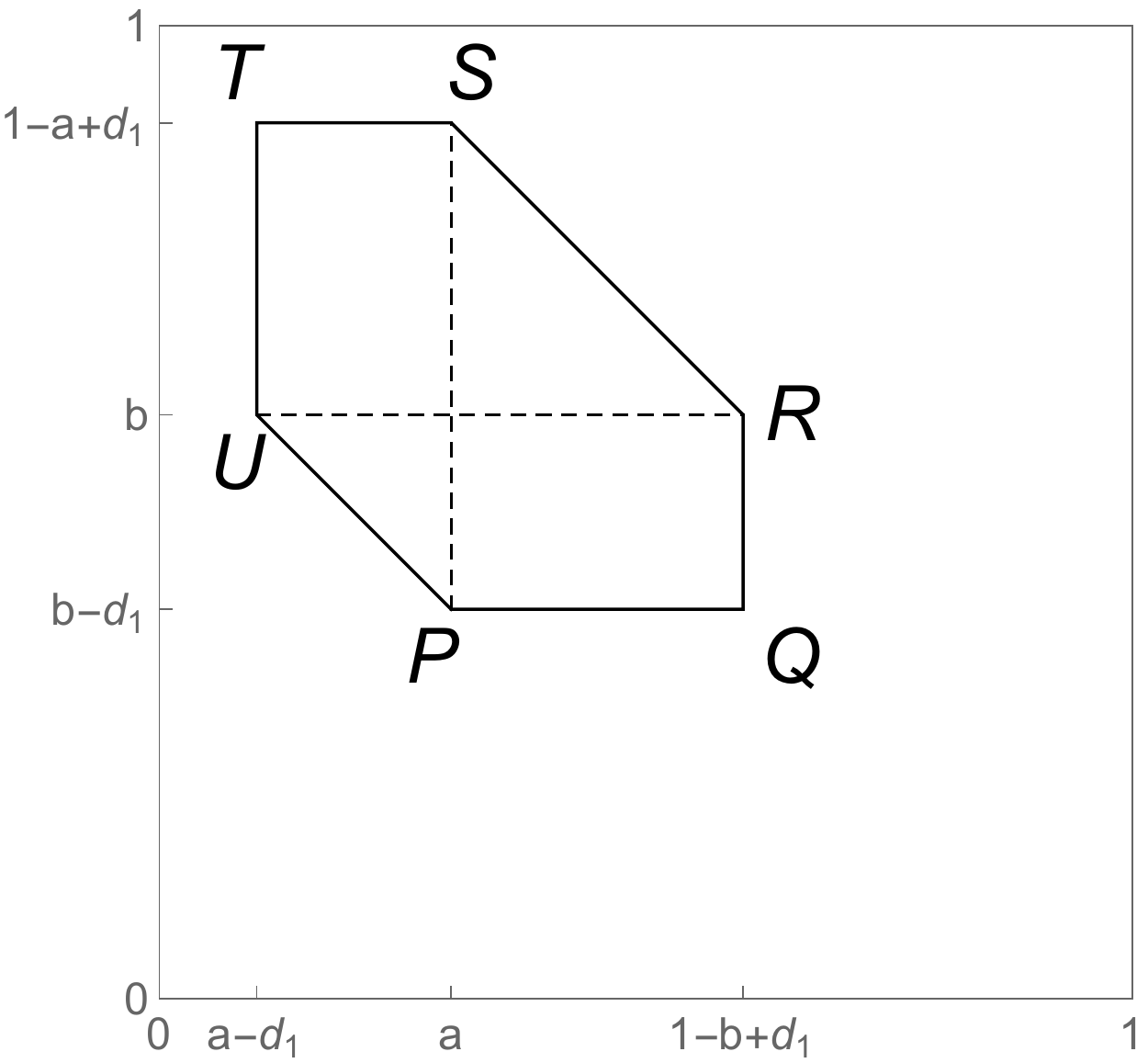}
            \caption{The hexagon $PQRSTU$ }\label{sestkotnik}
\end{figure}

    \noindent We consider several cases.
    \begin{enumerate}
        \item All the vertices lie above the main diagonal, which holds as soon as $Q$ lies above the main diagonal. It means that
            $b \geqslant d_1 + \frac12$. In this case in the first, the second and the third integral of
            \eqref{concordance_left} we have $M(u, v) = u$, while for the fourth integral we have $M(u,v)=u$ for $u\in [1-b+d_1,\frac12]$ and  $M(u,v)=1-u$ for $u\in                         [\frac12,1]$, so that
            $$ \cQ(M,\underline{C}^{(a,b)}_{c})  = 4 \int_0^{\frac12} u du + 4\int_{\frac12}^1 (1-u) du -1 = 0.$$
        \item $Q$ lies below the main diagonal, and all other vertices lie above the main diagonal, which holds as soon as $P$
            and $R$ lie above the main diagonal. It means that $\frac12(1+d_1) \leqslant b \leqslant d_1 + \frac12, a \leqslant b-d_1$. In this case again, we have $M(u, v) = u$ in
            the first, the second and the third integral of \eqref{concordance_left}, while in
            the fourth integral we have $M(u, v) = v$. Then it follows that
            $$ \cQ(M,\underline{C}^{(a,b)}_{c})  = 4 \int_0^{1-b+d_1} u du + 4\int_{1-b+d_1}^1 (1-u) du -1 = (2d_1+1-2b)^2.$$
        \item $P$ and $Q$ lie below the main diagonal, and all other vertices lie above the main diagonal. It means that
            $\frac12(1+d_1) \leqslant b \leqslant d_1 + \frac12, a \geqslant b-d_1$. In this case we have $M(u, v) = u$ in
            the first and the third integral of \eqref{concordance_left},  we have $M(u, v) = v$ in
            the fourth integral, while the second integral splits into two parts
            $$ \int_{a-d_1}^{a} M(u,a+b-d_1-u) du = \int_{a-d_1}^{\frac12(a+b-d_1)} udu + \int_{\frac12(a+b-d_1)}^{a} (a+b-d_1-u)du,$$
            so that
            \[\begin{split}
           \cQ(M,\underline{C}^{(a,b)}_{c})  &= 4 \int_0^{\frac12(a+b-d_1)} u du + 4\int_{\frac12(a+b-d_1)}^{a} (a+b-d_1-u)du + \\
                                       & \phantom{x} + 4 \int_a^{1-b+d_1} u du + 4\int_{1-b+d_1}^1 (1-u) du -1 = \\
                                 &= (a+b-1-d_1)(3b-3d_1-a-1).
            \end{split}\]
        \item $Q$ and $R$ lie below the main diagonal, and all other vertices lie above the main diagonal. It means that
            $b \leqslant \frac12(1+d_1), a \leqslant b-d_1$. In this case in
            the first and the second integral of \eqref{concordance_left} we have $M(u, v) = u$, in
            the fourth we have $M(u, v) = v$, and the third integral splits into two parts
            $$ \int_{a}^{1-b+d_1} M(u,1+d_1-u)du = \int_{a}^{\frac12(1+d_1)} udu + \int_{\frac12(1+d_1)}^{1-b+d_1} (1+d_1-u)du,$$
            so that
            \[\begin{split}
           \cQ(M,\underline{C}^{(a,b)}_{c})  &= 4 \int_0^{\frac12(1+d_1)} u du + 4\int_{\frac12(1+d_1)}^{1-b+d_1} (1+d_1-u)du + \\
                                       & \phantom{x} + 4\int_{1-b+d_1}^1 (1-u) du -1 = \\
                                 &= d_1(2+3d_1-4b).
            \end{split}\]
        \item $P$, $Q$ and $R$ lie below the main diagonal, and $S$, $T$ and $U$ lie above the main diagonal. It means that
            $d_1 \geqslant 2b-1, d_1 \geqslant 2a-1, d_1 \geqslant b-a$, and $d_1 \geqslant a-b$. {In this case in the first integral of \eqref{concordance_left} we have $M(u, v) = u$, in the fourth integral we have $M(u, v) = v$, while the second and the third integral split into two parts. We get
             \[\begin{split}
           \cQ(M,\underline{C}^{(a,b)}_{c})  &= 4 \int_0^{a-d_1} u du + 4\int_{a-d_1}^{\frac12(a+b-d_1)} u du + \\
                                       & \phantom{x} + 4\int_{\frac12(a+b-d_1)}^a (a+b-d_1-u) du + 4\int_a^{\frac12(1+d_1)} u du +\\
                                       & \phantom{x} + 4\int_{\frac12(1+d_1)}^{1-b+d_1} (1+d_1- u) du + 4\int_{1-b+d_1}^1 (1 - u) du - 1 =\\
                                 &= 2d_1(1+d_1-a-b)-(a-b)^2.
            \end{split}\]
            } 
        \item $S$ and $T$ lie above the main diagonal, and all other vertices lie below the main diagonal. This case is symmetric
            to case (d), only the roles of $a$ and $b$ are interchanged.
        \item $T$ and $U$ lie above the main diagonal, and all other vertices lie below the main diagonal. This case is symmetric
            to case (c).
        \item $T$ lies above the main diagonal, and all other vertices lie below the main diagonal. This case is symmetric
            to case (b).
        \item All the vertices lie below the main diagonal. This case is symmetric to case (a).
    \end{enumerate}
    \item 
        {This is a consequence of \cite[Theorem 5.1]{FuMcCSch}.}
\end{enumerate}
\end{proof}

For the next Proposition recall that {$d_1=d_{1,c}^{(a,b)}$ and $d_2=d_{2,c}^{(a,b)}$ are} 
given by {\eqref{d_1} and} (\ref{d_2}), {respectively}.

\begin{proposition} \label{prop3}
Let $(a,b)\in\II^2$ and $0\leqslant c\leqslant\min\{a,b,1-a,1-b\}$. For copulas $\overline{C}^{(a,b)}_{c}$ it holds:
\begin{enumerate}
\item $\cQ(W, \overline{C}^{(a,b)}_{c}) = $ \\ $=\left\{ \begin{array}{ll}
        0;                                & \text{if } a+b-\frac12 \leqslant d_2, \vspace{1mm}\\
        -(2a+2b-2d_2-1)^2;                & \text{if } \max \{2a+b-1,a+2b-1\} \leqslant d_2 \leqslant a+b-\frac12, \vspace{1mm}\\
        -(4a+3b-3d_2-2)(b-d_2);           & \text{if } a+2b-1\leqslant d_2 \leqslant 2a+b-1, \vspace{1mm}\\
        -(3a+4b-3d_2-2)(a-d_2);           & \text{if } 2a+b-1\leqslant d_2 \leqslant a+2b-1, \vspace{1mm}\\
        (a-1)^2+(b-1)^2+2d_2(a+b-d_2)-1;  & \text{if } d_2 \leqslant \min \{ 1-a,1-b,2a+b-1,a+2b-1\}, \vspace{1mm}\\
        (a-d_2)(a+3d_2-2);                & \text{if } 1-b \leqslant d_2 \leqslant 1-a, \vspace{1mm}\\
                (b-d_2)(b+3d_2-2);                & \text{if } 1-a \leqslant d_2 \leqslant 1-b, \vspace{1mm}\\
        -(1-2d_2)^2;                      & \text{if } \max \{1-a,1-b \} \leqslant d_2 \leqslant \frac12, \vspace{1mm}\\
                0;                                & \text{if } \frac12 \leqslant d_2,
                \end{array} \right.$
\item $\cQ(\Pi, \overline{C}^{(a,b)}_{c}) = \frac13-2(a+b-2d_2)(a-d_2)(b-d_2),$
\item $\cQ(M, \overline{C}^{(a,b)}_{c}) = 1 - 4(a-d_2)(b-d_2),$
\item $\cQ(\overline{C}^{(a,b)}_{c}, \overline{C}^{(a,b)}_{c}) = 1 - 4(a-d_2)(b-d_2).$
\end{enumerate}
\end{proposition}

\begin{proof}
{To prove the proposition we use Property (Q4) of the concordance function, Lemma \ref{lem1} and Proposition \ref{prop2}.  More precisely, equalities (Q4), \eqref{rel Cabc}, Proposition \ref{prop2}, and \eqref{rel d12} are applied in succession in each of the following three calculations:
\begin{eqnarray*}
    \cQ(\Pi, \overline{C}^{(a,b)}_{c})
    &=& - \cQ(\Pi, \underline{C}^{(b,1-a)}_{c}) \\
    &=& \frac13 - 2d_{1,c}^{(b,1-a)}\left(1-a-b+d_{1,c}^{(b,1-a)}\right)\left(1-a-b+2d_{1,c}^{(b,1-a)}\right)\\
    &=& \frac13-2(a-d_2)(b-d_2)(a+b-2d_2),
\end{eqnarray*}
\begin{eqnarray*}\cQ(M, \overline{C}^{(a,b)}_{c}) &=& - \cQ(W, \underline{C}^{(b,1-a)}_{c}) \\
    &=& 1 - 4d_{1,c}^{(b,1-a)}\left(1-a-b+d_{1,c}^{(b,1-a)}\right)\\
    &=& 1-4(a-d_2)(b-d_2),
\end{eqnarray*}
 and
\begin{eqnarray*}\cQ(\overline{C}^{(a,b)}_{c}, \overline{C}^{(a,b)}_{c}) &=& - \cQ(\underline{C}^{(b,1-a)}_{c}, \underline{C}^{(b,1-a)}_{c}) \\
    &=& 1 - 4d_{1,c}^{(b,1-a)}\left(1-a-b+d_{1,c}^{(b,1-a)}\right)\\
    &=&1-4(a-d_2)(b-d_2).
\end{eqnarray*}

Thus we proved (2), (3) and (4).

To prove (1) we use the same arguments, but a more detailed case by case analysis is required. We omit the details.    }
\end{proof}

It turns out that the results obtained in Propositions \ref{prop2} and \ref{prop3} are symmetric with respect to the main diagonal and to the counter-diagonal. The proofs {follow directly from Lemma \ref{lem1} using Properties (Q1), (Q3) and (Q4) of the concordance function.}

\begin{proposition} \label{prop4}
Let $(a,b)\in\II^2$ and $0\leqslant c\leqslant\min\{a,b,1-a,1-b\}$. Then the following holds for any $D\in \{W, \Pi, M\}$:
\begin{enumerate}
    \item $\cQ(D,\underline{C}^{(a,b)}_{c})=\cQ(D,\underline{C}^{(b,a)}_{c})$ and $\cQ(D,\overline{C}^{(a,b)}_{c})=\cQ(D,\overline{C}^{(b,a)}_{c})$.
    \item $\cQ(D,\underline{C}^{(a,b)}_{c})=\cQ(D,\underline{C}^{(1-a,1-b)}_{c})$ and $\cQ(D,\overline{C}^{(a,b)}_{c})=\cQ(D,\overline{C}^{(1-a,1-b)}_{c})$.
    \item { $\cQ(D,\underline{C}^{(a,b)}_{c})=-\cQ(D^{\sigma_1},\overline{C}^{(b,1-a)}_{c})$ and $\cQ(D,\overline{C}^{(a,b)}_{c})=-\cQ(D^{\sigma_1},\underline{C}^{(b,1-a)}_{c})$.}
    \item $\cQ(\underline{C}^{(a,b)}_{c},\underline{C}^{(a,b)}_{c})=\cQ(\underline{C}^{(b,a)}_{c},\underline{C}^{(b,a)}_{c})$ and $\cQ(\overline{C}^{(a,b)}_{c},\overline{C}^{(a,b)}_{c})=\cQ(\overline{C}^{(a,b)}_{c},\overline{C}^{(a,b)}_{c})$.
    \item $\cQ(\underline{C}^{(a,b)}_{c},\underline{C}^{(a,b)}_{c})=\cQ(\underline{C}^{(1-a,1-b)}_{c},\underline{C}^{(1-a,1-b)}_{c})$ and $\cQ(\overline{C}^{(a,b)}_{c},\overline{C}^{(a,b)}_{c})=\cQ(\overline{C}^{(1-a,1-b)}_{c},\overline{C}^{(1-a,1-b)}_{c})$.
   \item { $\cQ(\underline{C}^{(a,b)}_{c},\underline{C}^{(a,b)}_{c})=-\cQ(\overline{C}^{(b,1-a)}_{c},\overline{C}^{(b,1-a)}_{c})$ and $\cQ(\overline{C}^{(a,b)}_{c},\overline{C}^{(a,b)}_{c})=-\cQ(\underline{C}^{(b,1-a)}_{c},\underline{C}^{(b,1-a)}_{c})$.}
   \end{enumerate}
\end{proposition}

We now prove the crucial lemma which will enable us to find the relations between measures of concordance and asymmetry.

Let $C\in \Cmt$ be any copula with $\mu_\infty(C)=m$ and $\kappa \in \{ \rho, \tau, \phi, \gamma \}$ a measure of concordance. By the definition of asymmetry there exists a pair $(a_0,b_0)\in\II^2$ such that
\[
    \left|C(a_0,b_0)-C(b_0,a_0)\right|=m.
\]
Here we may choose with no loss that $a_0\leqslant b_0$ {due to the fact that the considered measures of concordance are invariant with respect to the symmetry by Proposition \ref{prop4}(1).} However, the expression between vertical bars may be positive or negative. In the case of negative value we replace $C$ with $C^t$, hence we may assume that $C(a_0,b_0)-C(b_0,a_0)=m \geqslant 0$. It follows from Theorem \ref{th1} that
    $$\underline{C}^{(a_0,b_0)}_{m}\leqslant C\leqslant \overline{C}^{(a_0,b_0)}_{m}$$
and by monotonicity of $\kappa$ we have that
\[
    \kappa(\underline{C}^{(a_0,b_0)}_{m})\leqslant \kappa(C)\leqslant \kappa(\overline{C}^{(a_0,b_0)}_{m}).
\]
We fix $m$ and we denote by $\Delta RST$ a triangle within the square $\II^2$ such that
\[
    \Delta RST=\{(a,b)\in\II^2\,;\,a\leqslant b, d_\Cmt^*(a,b) \geqslant m\}.
\]
Since $C(a_0,b_0)-C(b_0,a_0)=m \geqslant 0$ it must be that $(a_0,b_0) \in \Delta RST$. It can be verified easily that
\[
    \Delta RST=\{(a,b)\in\II^2\,; a \geqslant m, b \leqslant 1-m, b-a \geqslant m \}
\]
with the vertices $R(m,2m)$, $S(1-2m,1-m)$ and $T(m,1-m)$. Denote by $U(\frac12(1-m),\frac12(1+m))$ the center of the edge $RS$, see Figure \ref{trikot}.
Define $\Delta_m:=\Delta RUT.$

\begin{figure}[h]
            \includegraphics[width=5cm]{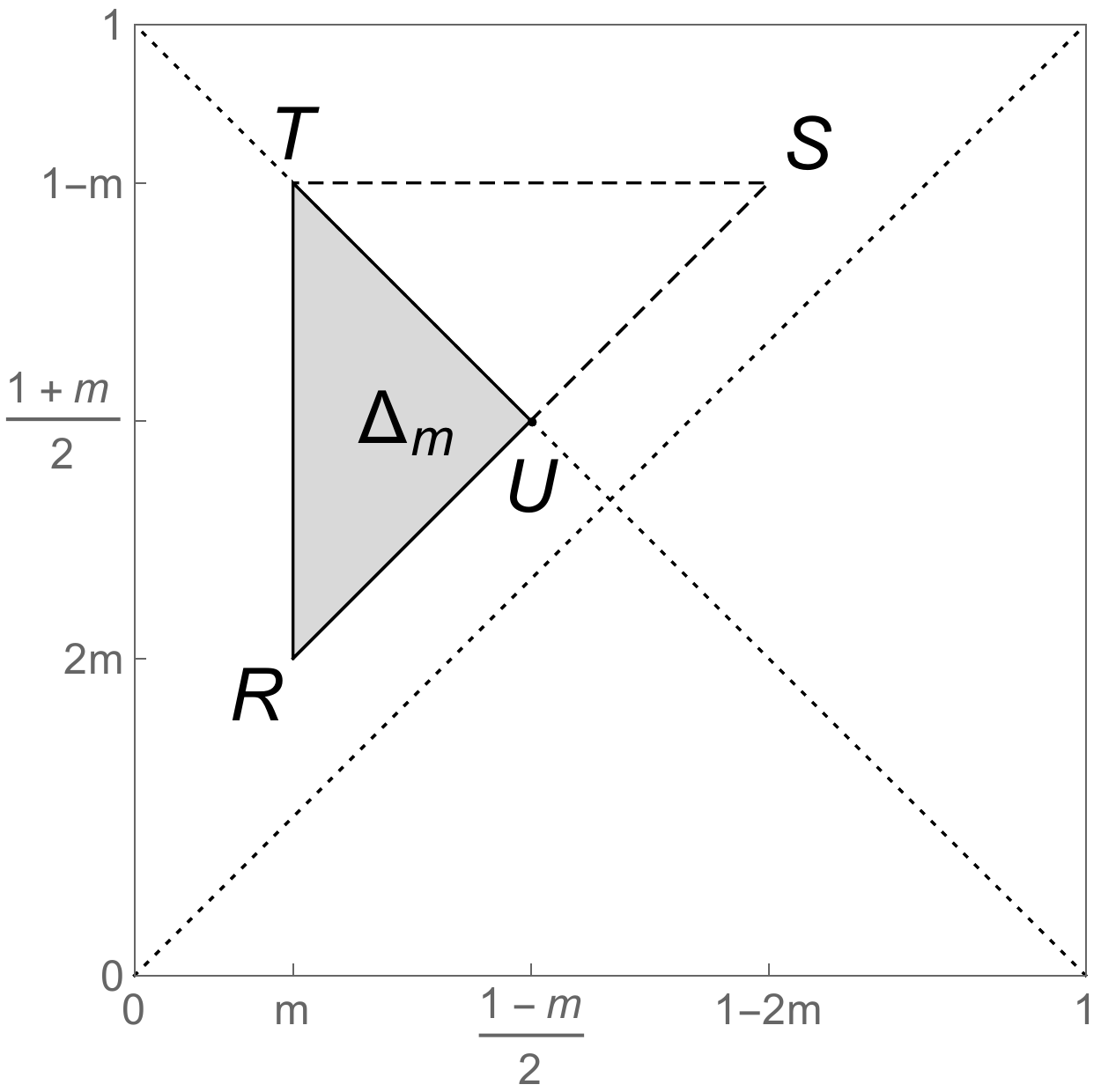}
            \caption{The triangle $\Delta_m$ }\label{trikot}
\end{figure}

It follows that
\[
\kappa(C)\in \left[\min_{(a,b)\in\Delta RST}\kappa(\underline{C}^{(a,b)}_{m}), \max_{(a,b)\in\Delta RST}\kappa(\overline{C}^{(a,b)}_{m})\right].
\]
Due to Proposition \ref{prop4}, the minimum and the maximum can be taken over $\Delta_m$ instead of $\Delta RST$.

\begin{lemma}
Let $C\in \Cmt$ be any copula with $\mu_\infty(C)=m$ and $\kappa \in \{ \rho, \tau, \phi, \gamma \}$ a measure of concordance. Then
\begin{equation} \label{eq:meji}
\kappa(C)\in \left[\min_{(a,b)\in\Delta_m}\kappa(\underline{C}^{(a,b)}_{m}), \max_{(a,b)\in\Delta_m}\kappa(\overline{C}^{(a,b)}_{m})\right].
\end{equation}
\end{lemma}

Note that either $\max$ in \eqref{eq:meji} can be replaced by $\min$ or $\min$ by $\max$ if we use relation \eqref{rel Cabc}, and Properties (C2) and (C4). For instance, we have
\begin{align}
\max_{(a,b)\in\Delta_m}\kappa(\overline{C}^{(a,b)}_{m})&=\max_{(a,b)\in\Delta_m}\left(-\kappa(\underline{C}^{(1-b,a)}_{m})\right) \nonumber \\
&=-\min_{(a,b)\in\Delta_m}\left(\kappa(\underline{C}^{(a,1-b)}_{m})\right) \nonumber \\
&=-\min_{(a,b)\in\Delta_m^{\sigma_1}}\left(\kappa(\underline{C}^{(a,b)}_{m})\right),\label{max-to-min}
\end{align}
where $\Delta_m^{\sigma_1}=\{(a,b)\in\II^2\,; a \geqslant m, b-a \geqslant 0, a+b \leqslant 1-m \}$. Observe however, that copulas $\underline{C}^{(a,b)}_{m}$ appearing in the minimum of \eqref{max-to-min} do not need to satisfy \eqref{asymmetry point}. We do not consider such copulas in the proofs of the following sections.

\section{Relations between Spearman's rho and asymmetry }\label{sec:rho}

To find the relations between Spearman's rho and asymmetry, we compute minimum and maximum from \eqref{eq:meji} for $\kappa=\rho$.

\begin{lemma}\label{lem:rho}
  \begin{description}
    \item[(a)] $\displaystyle\min_{(a,b)\in\Delta_m}\rho(\underline{C}^{(a,b)}_{m})= \rho(\underline{C}^{(m,1-m)}_{ m})=12m^3-1$
    \item[(b)] $\displaystyle\max_{(a,b)\in\Delta_m}\rho(\overline{C}^{(a,b)}_{c})= \rho(\overline{C}^{(m,2m)}_{ m})=1-36m^3$.
  \end{description}
\end{lemma}

\begin{proof}
{\bf (a): }
Recall that $\rho(C)=3\cQ(C,\Pi)$.
Take $(a,b) \in \Delta_m.$
Since $a+b \leqslant 1$, we have $W(a,b)=0$, so $d_1=m+W(a,b)=m$ and by Proposition \ref{prop2} we get that
$$\rho(\underline{C}^{(a,b)}_{m})=3 \cQ(\Pi, \underline{C}^{(a,b)}_{m})= 6m(1-a-b+m)(1-a-b+2m)-1.$$
We look at this expression as a quadratic function of $x=a+b$, with $x \in [3m,1]$. Then
$$f(x):=\rho(\underline{C}^{(a,b)}_{m})=6m (x^2- (3m+2)x + 2m^2+3m+1)-1.$$
Its minimum is attained at a point $x=\frac32 m+1 \notin [3m,1]$, hence our optimal value appears at $x=1=a+b$. Thus $\rho(\underline{C}^{(a,b)}_{m})$ takes its minimal value $f(1)=12m^3-1$ on the entire line segment $UT$, in particular at $(a,b)=(m,1-m)$.

{\bf (b):}
Since the triangle $\Delta_m$ lies above the main diagonal, we have $M(a,b)=a$, thus $d_2=M(a,b)-m=a-m$. By Proposition \ref{prop3}
$$\rho(\overline{C}^{(a,b)}_{m})=3Q (\Pi,\overline{C}^{(a,b)}_{m})=1-6(a+b-d_2)(a-d_2)(b-d_2)=1-6m(b-a+2m)(b-a+m).$$
We look at this expression as a quadratic function of $x=b-a$, with $x \in [m,1-2m]$. Then
$$g(x):=\rho(\overline{C}^{(a,b)}_{m})=1-6m (x^2 +3 m x + 2m^2).$$
Its maximum is attained at a point $x=-\frac32 m \notin [m,1-2m]$, hence our optimal value appears at $x=m=b-a$. Thus $\rho(\overline{C}^{(a,b)}_{m})$ takes its maximal value $g(m)=1-36m^3$ on the entire line segment $RS$, in particular at $(a,b)=(m,2m)$.
\end{proof}

We collect our findings in the main theorem of this section.

\begin{theorem}\label{th:rho}
Let $C\in \Cmt$ be any copula with $\mu_\infty(C)=m$. Then $$\rho(C)\in[12m^3-1,1-36m^3]$$ and any value from this interval can be attained. In particular, if copula $C$ is symmetric, then $\rho(C)$ can take any value in $[-1,1]$. If copula $C$ takes the maximal possible asymmetry $\mu_\infty(C)=\frac13$, then $\rho(C) \in [-\frac59, -\frac13]$.
\end{theorem}

\begin{proof}
It remains to be proved only that any value from the interval $[12m^3-1,1-36m^3]$ can be attained. First notice that 
a convex combination
$C = t\underline{C}^{(a,b)}_{m} + (1-t)\overline{C}^{(a,b)}_{m}$ is a copula with asymmetry $m$ for any $t \in [0, 1]$. { Since $\rho$ is polynomial in $t$ \cite[p. 1778]{EdTa} } it then follows that the image of $\rho$
on the 
 set of all copulas with asymmetry $m$ is a closed interval.
\end{proof}

Using inverse functions, the following corollary immediately follows. Figure \ref{fig-rho} shows these relations.

\begin{corollary}\label{cor:rho}
  If $\rho(C)=\rho$, then
  \[
     0 \leqslant \mu_\infty(C) \leqslant \left\{
           \begin{array}{ll}
             \sqrt[3]{\displaystyle\frac{1+\rho}{12}}; & \hbox{$-1\leqslant \rho\leqslant-\frac59$,} \vspace{1mm}\\
             \displaystyle\frac{1}{3}; & \hbox{$-\frac59\leqslant \rho\leqslant -\frac13$,} \vspace{1mm}\\
             \sqrt[3]{\displaystyle\frac{1-\rho}{36}}; & \hbox{$-\frac13\leqslant \rho\leqslant1$,}
           \end{array}
         \right.
  \]
    and the bounds are attained.
\end{corollary}

\begin{figure}[h]
            \includegraphics[width=4cm]{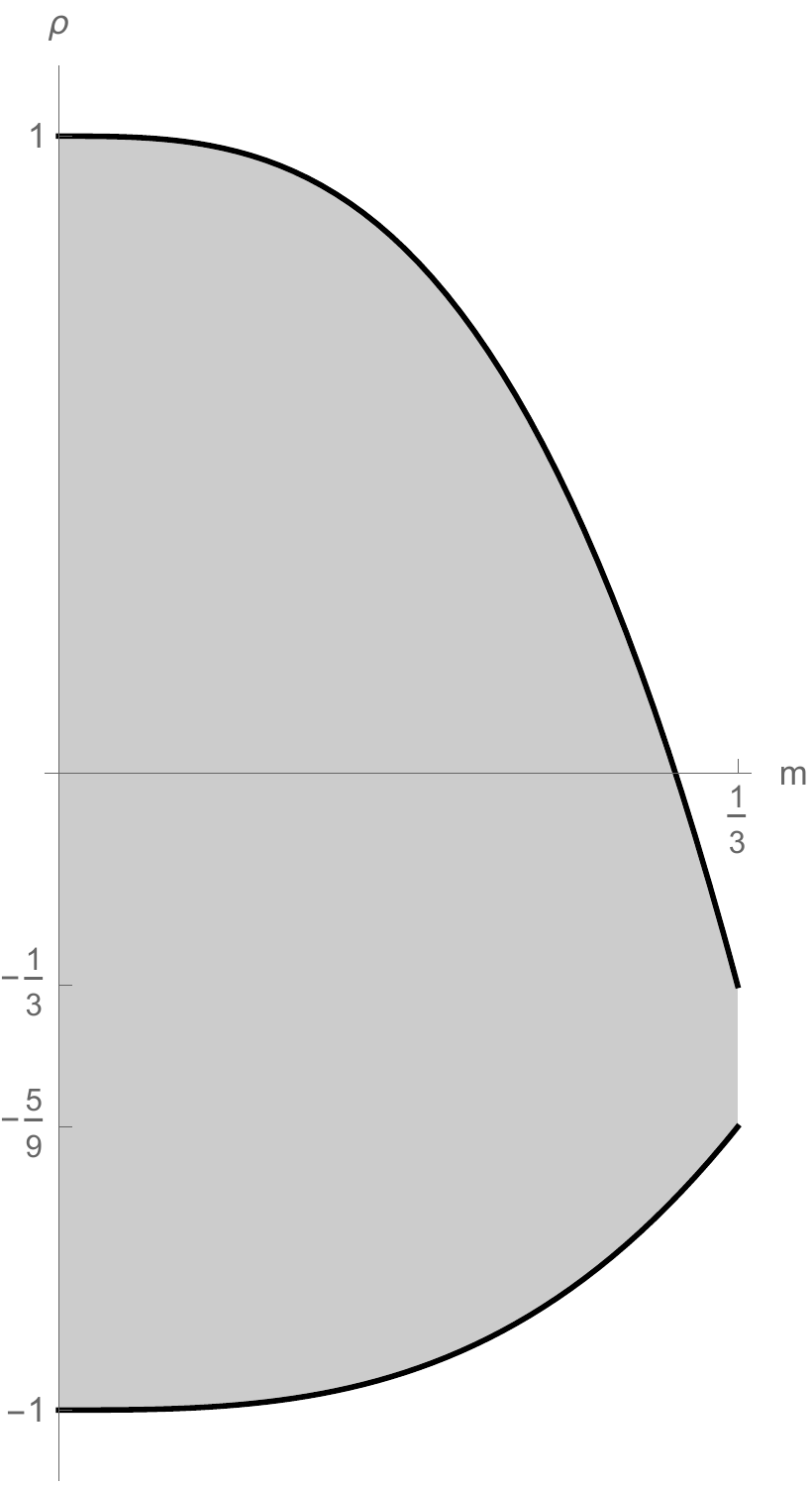} \hfil \includegraphics[width=6cm]{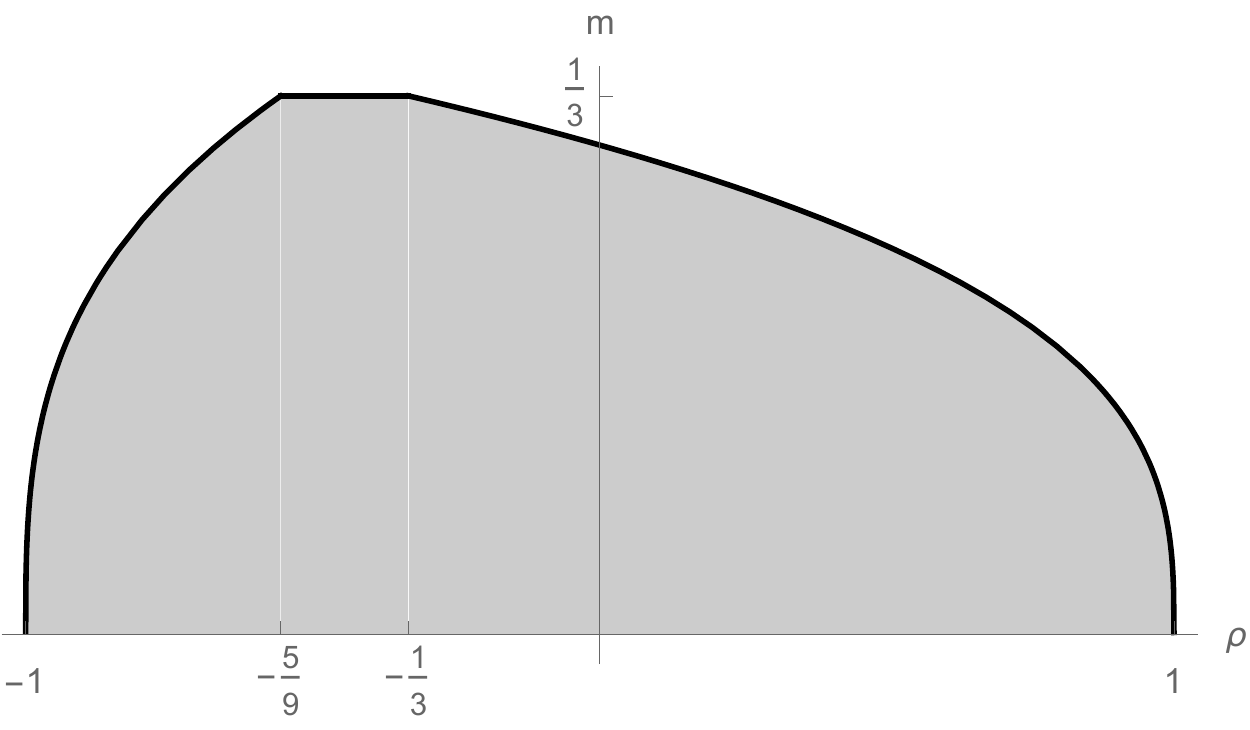}
            \caption{Relations between Spearman's rho and asymmetry }\label{fig-rho}
\end{figure}

\section{Relations between Kendall's tau and asymmetry }\label{sec:tau}

To find the relations between Kendall's tau and asymmetry, we compute minimum and maximum from \eqref{eq:meji} for $\kappa=\tau$.

\begin{lemma}\label{lem:tau}
  \begin{description}
    \item[(a)] $\displaystyle\min_{(a,b)\in\Delta_m}\tau(\underline{C}^{(a,b)}_{m})= {\tau(\underline{C}^{(m,1-m)}_{m})}=4m^2-1$
    \item[(b)] $\displaystyle\max_{(a,b)\in\Delta_m}\tau(\overline{C}^{(a,b)}_{m})= \tau(\overline{C}^{(m,2m)}_{m}) =1-8m^2$.
  \end{description}
\end{lemma}

\begin{proof}
{\bf (a): }
Recall that $\tau(C)=\cQ(C,C)$. For $(a,b) \in \Delta_m$
we have by Proposition \ref{prop2} that
$$\tau(\underline{C}^{(a,b)}_{m})= 4m(1-a-b+m)-1.$$
This is a linear function of $x=a+b$ for $x \in [3m,1]$, which takes its minimum at {$x=1$}
, which corresponds to 
{$a+b=1$}.
Thus $\tau(\underline{C}^{(a,b)}_{m})$ takes its minimal value 
{$4m^2-1$ on the entire line segment $UT$, in particular at
$(a,b) = (m,1-m)$.}

{\bf (b):}
Since the triangle $\Delta_m$ lies above the main diagonal, we have $d_2 = a-m$. We use Proposition \ref{prop3} to obtain that
$$\tau(\overline{C}^{(a,b)}_{m})=1-4(a-d_2)(b-d_2)=1-4m(b-a+m).$$
This a linear function of $x=b-a$ for $x \in [m,1-2m]$, which takes its maximum at $x=m$, which corresponds to 
{$b=a+m$}.
Thus $\tau(\overline{C}^{(a,b)}_{m})$ takes its maximal value 
{on the entire line segment $RS$}, in particular at 
{$(a,b) = (m,2m)$}, where $\tau(\overline{C}^{(m,2m)}_{m})=1-8m^2$.
\end{proof}

Using the same argument as in section \ref{sec:rho}, we get the main theorem of this section and its corollary.
Figure \ref{fig-tau} shows the relations between Kendall's tau and asymmetry.

\begin{theorem}\label{th:tau}
Let $C\in \Cmt$ be any copula with $\mu_\infty(C)=m$. Then $$\tau(C)\in[4m^2-1,1-8m^2]$$ and any value from this interval can be attained. In particular, if copula $C$ is symmetric, then $\rho(C)$ can take any value in $[-1,1]$. If copula $C$ takes the maximal possible asymmetry $\mu_\infty(C)=\frac13$, then $\rho(C) \in [-\frac59, \frac19]$.
\end{theorem}

\begin{corollary}\label{cor:tau}
  If $\tau(C)=\tau$, then
  \[
     0 \leqslant \mu_\infty(C) \leqslant \left\{
           \begin{array}{ll}
             \sqrt{\displaystyle\frac{1+\tau}{4}}; & \hbox{$-1\leqslant \tau \leqslant-\frac59$,} \vspace{1mm}\\
             \displaystyle\frac{1}{3}; & \hbox{$-\frac59\leqslant \tau\leqslant \frac19$,} \vspace{1mm}\\
             \sqrt{\displaystyle\frac{1-\tau}{8}}; & \hbox{$\frac19\leqslant \tau \leqslant1$,}
           \end{array}
         \right.
  \]
    and the bounds are attained.
\end{corollary}

\begin{figure}[h]
            \includegraphics[width=4cm]{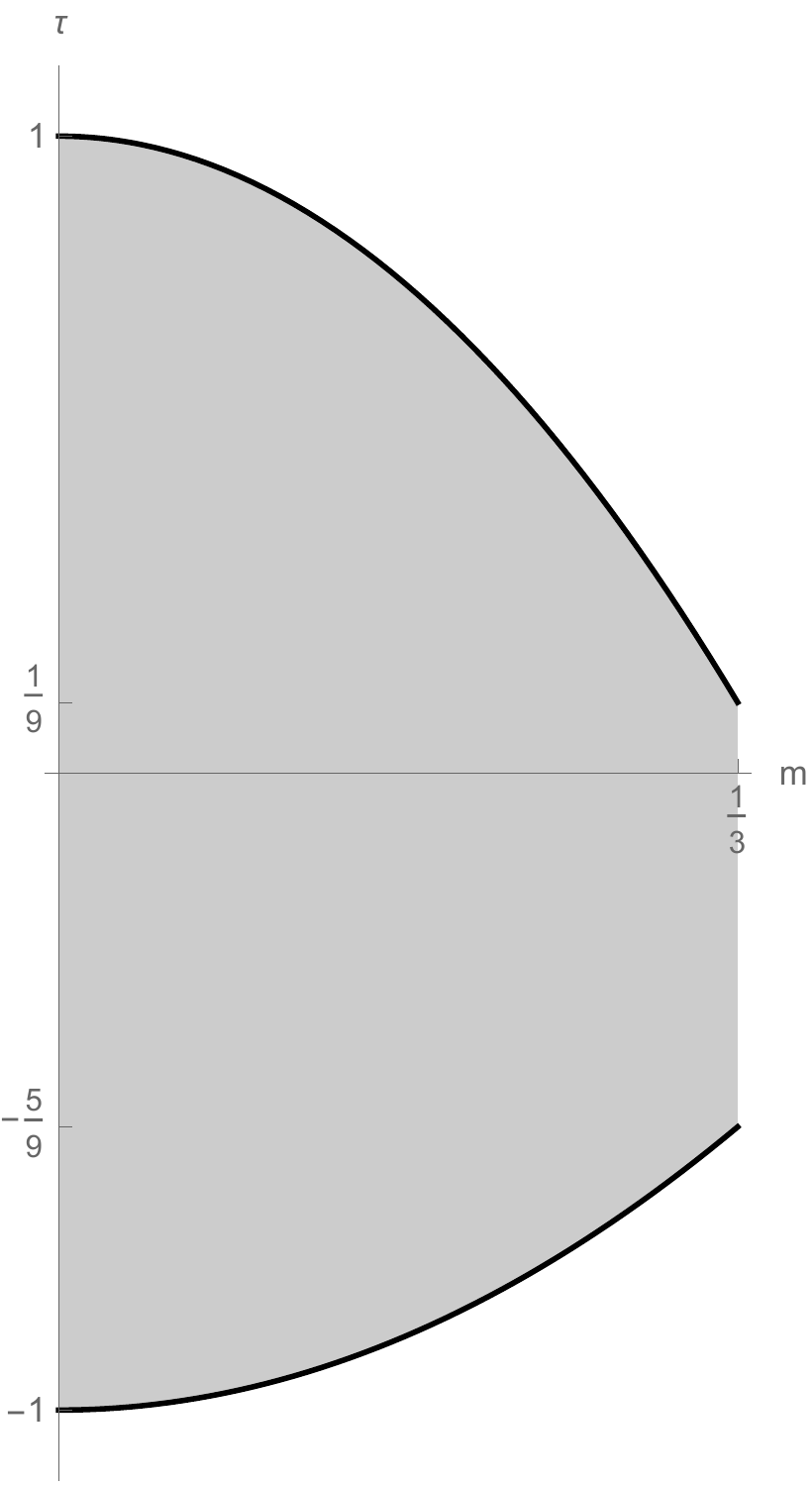} \hfil \includegraphics[width=6cm]{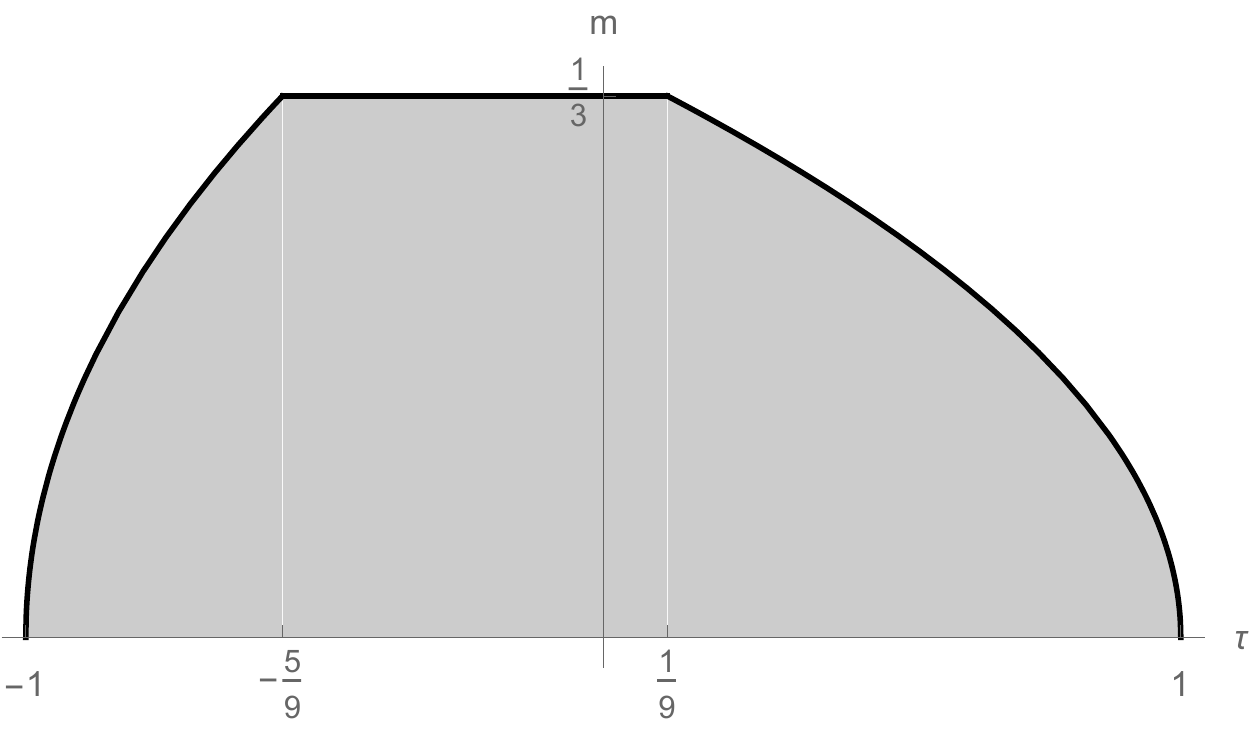}
            \caption{Relations between Kendall's tau and asymmetry }\label{fig-tau}
\end{figure}

\section{Relations between Spearman's footrule  and asymmetry }\label{sec:phi}

To find the relations between Spearman's footrule and asymmetry, we compute minimum and maximum from \eqref{eq:meji} for $\kappa=\phi$. Recall that
$$\phi(C)=\textstyle\frac12\left(3\cQ(M,C)-1\right).$$

\begin{lemma}\label{lem:footrule}
  \begin{description}
    \item[(a)] $\displaystyle\min_{(a,b)\in\Delta_m}\phi(\underline{C}^{(a,b)}_{m})= \phi(\underline{C}^{(m,1-m)}_{ m})=\left\{ \begin{array}{ll}
        -\frac12;    & \text{if } 0 \leqslant m \leqslant \frac14, \vspace{1mm}\\
        24m^2-12m+1; & \text{if } \frac14 \leqslant m \leqslant \frac13. \\
                \end{array} \right.$
    \item[(b)] $\displaystyle\max_{(a,b)\in\Delta_m}\phi(\overline{C}^{(a,b)}_{m})= \phi(\overline{C}^{(m,2m)}_{m})=1-12m^2$.
  \end{description}
\end{lemma}

\begin{proof}
{\bf (a): }
For $(a,b) \in \Delta_m$ in the expression for $\cQ(M, \underline{C}^{(a,b)}_{m})$ only three out of nine cases are possible, namely, the first, the second and the fourth.
In the third case, when the vertex $P$ of the hexagon is below the main diagonal and $R$ is above it (see Figure \ref{sestkotnik}),
the central point $(a, b)$ is above the counter-diagonal. In the remaining four cases we have the central point below the main diagonal. Since in the triangle $\Delta_m$, we have $d_1 = W(a, b) +m$ and $b-a \geqslant m$ the expression for  $\cQ(M, \underline{C}^{(a,b)}_{m})$
simplifies to
\begin{equation} \label{cQ(M,C1)}
\cQ(M, \underline{C}^{(a,b)}_{m}) = \left\{ \begin{array}{ll}
        0;                        & \text{if } m + \frac12 \leqslant b, \vspace{1mm}\\
        (2m+1-2b)^2;              & \text{if } \frac12(1+m) \leqslant b \leqslant m + \frac12, \vspace{1mm}\\
        m(2+3m-4b);               & \text{if } b \leqslant \frac12(1+m).  \\
                \end{array} \right.
\end{equation}
For fixed $m$ it depends only on $b$, and we need to minimize it for $b \in [2m, 1-m]$.
Notice that for big values of $m$ (i.e. close to $\frac13$) not all of three cases are possible.
Regardless of that, the left part, when $b \leqslant \frac12(1+m)$, is linear and decreasing in $b$, the middle part,
when $\frac12(1+m) \leqslant b \leqslant m + \frac12$ is quadratic and also decreasing in $b$, since
$\frac{d}{db}(2m+1-2b)^2 = -4(2m+1-2b) \leqslant 0$, and the right part is constant. So, $\phi(\underline{C}^{(a,b)}_{m})$ takes its minimal value
for $b= 1-m$, i.e. at the vertex $T$, where
$$\phi(\underline{C}^{(m,1-m)}_{m}) = -\textstyle\frac12 + \frac32 \left\{ \begin{array}{ll}
        0;        & \text{if } 0 \leqslant m \leqslant \frac14, \vspace{1mm}\\
        (4m-1)^2; & \text{if } \frac14 \leqslant m \leqslant \frac13, \\
                \end{array} \right.=\left\{ \begin{array}{ll}
        -\frac12;    & \text{if } 0 \leqslant m \leqslant \frac14, \vspace{1mm}\\
        24m^2-12m+1; & \text{if } \frac14 \leqslant m \leqslant \frac13. \\
                \end{array} \right.
                $$
{\bf (b):}
Notice that $\cQ(M, \overline{C}^{(a,b)}_{c}) = \cQ(\overline{C}^{(a,b)}_{c}, \overline{C}^{(a,b)}_{c})$, so $\phi(\overline{C}^{(a,b)}_{m}) = \frac12\left(3\tau(\overline{C}^{(a,b)}_{m})-1\right)$. Thus $\phi(\overline{C}^{(a,b)}_{m})$ takes its maximal value at the same 
{points} as $\tau(\overline{C}^{(a,b)}_{m})$, i.e. 
{the line segment $RS$, in particular at $(2m,m)$, } where $\phi(\overline{C}^{(m,2m)}_{m})=1-12m^2$.
\end{proof}

There is an intuition lying behind our computations for the case $\underline{C}^{(a,b)}_{c}$ to follow. Our copula is basically $W$ with a bump.
Since $\cQ(M,\underline{C}^{(a,b)}_{c}) = \cQ(\underline{C}^{(a,b)}_{c},M)$, we need to minimize the integral of $\underline{C}^{(a,b)}_{c}$ over the main diagonal. So, we need to push the bump as far towards northwest as possible.

The main theorem and its corollary now follow. Figure \ref{fig-phi} shows the relations between Spearman's footrule and asymmetry.

\begin{theorem}\label{th:phi}
Let $C\in \Cmt$ be any copula with $\mu_\infty(C)=m$. Then
\[
1-12m^2\geqslant\phi(C)\geqslant\left\{
                                  \begin{array}{ll}
                                    -\frac12; & \hbox{if $0 \leqslant m\leqslant\frac14$,} \vspace{1mm}\\
                                    24m^2-12m+1; & \hbox{if $\frac14 \leqslant m\leqslant \frac13$,}
                                  \end{array}
                                \right.
\]
and the bounds are attained. In particular, if copula $C$ is symmetric, then $\phi(C)$ can take any value in $[-\frac12,1]$. If copula $C$ takes the maximal possible asymmetry $\mu_\infty(C)=\frac13$, then $\phi(C) = -\frac13$.
\end{theorem}

\begin{corollary}\label{cor:phi}
  If $\phi(C)=\phi$, then
  \[
    0\leqslant\mu_\infty(C)\leqslant\left\{
           \begin{array}{ll}
             \displaystyle\frac{3+\sqrt{3+6\phi}}{12}; & \hbox{if $\frac12 \leqslant\phi\leqslant -\frac13$,} \vspace{1mm}\\
             \displaystyle\sqrt{\frac{1-\phi}{12}}; & \hbox{if $-\frac13\leqslant \phi \leqslant 1$,}
           \end{array}
         \right.
  \]
and the bounds are attained.
\end{corollary}

\begin{figure}[h]
            \includegraphics[width=4cm]{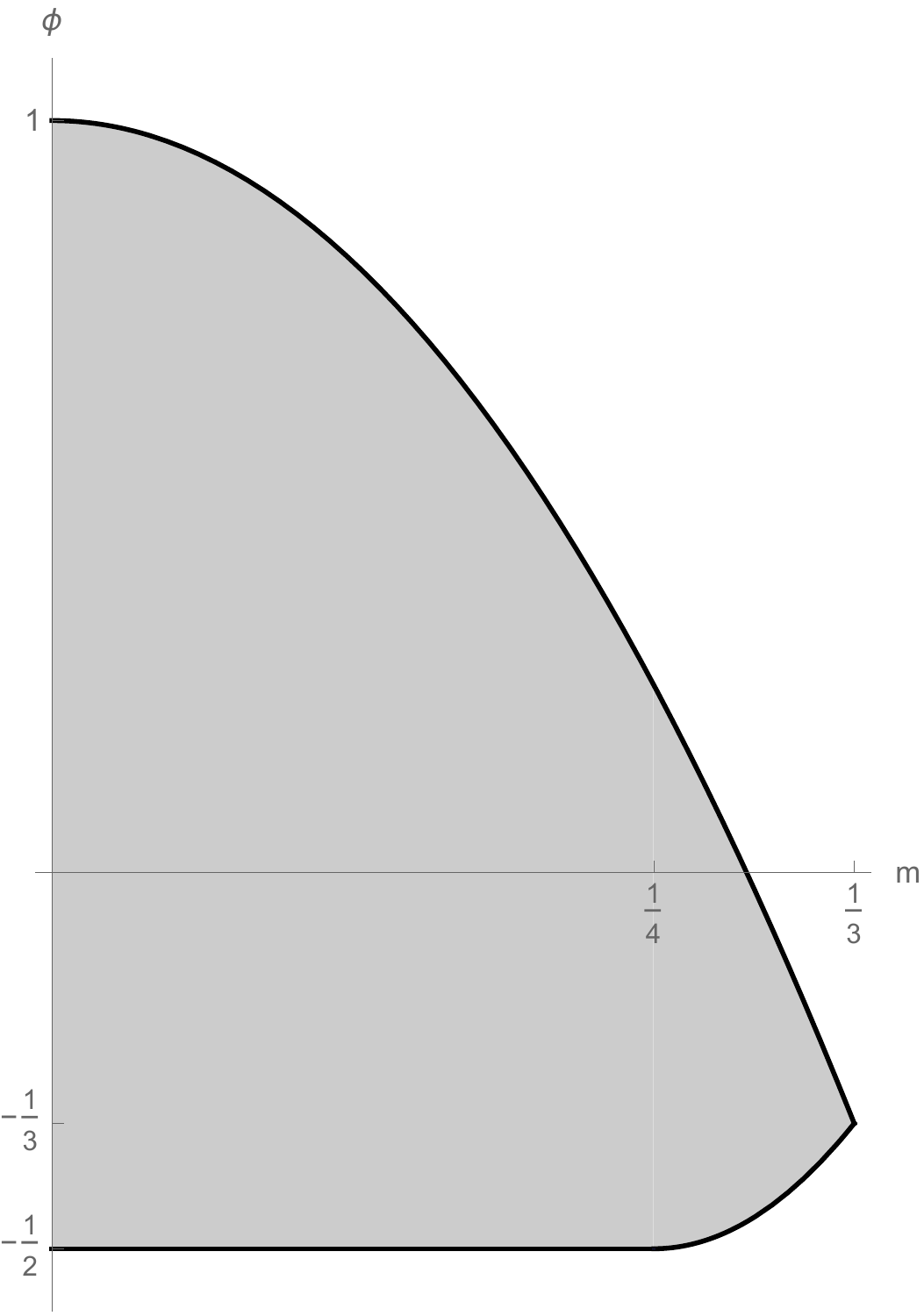} \hfil \includegraphics[width=5cm]{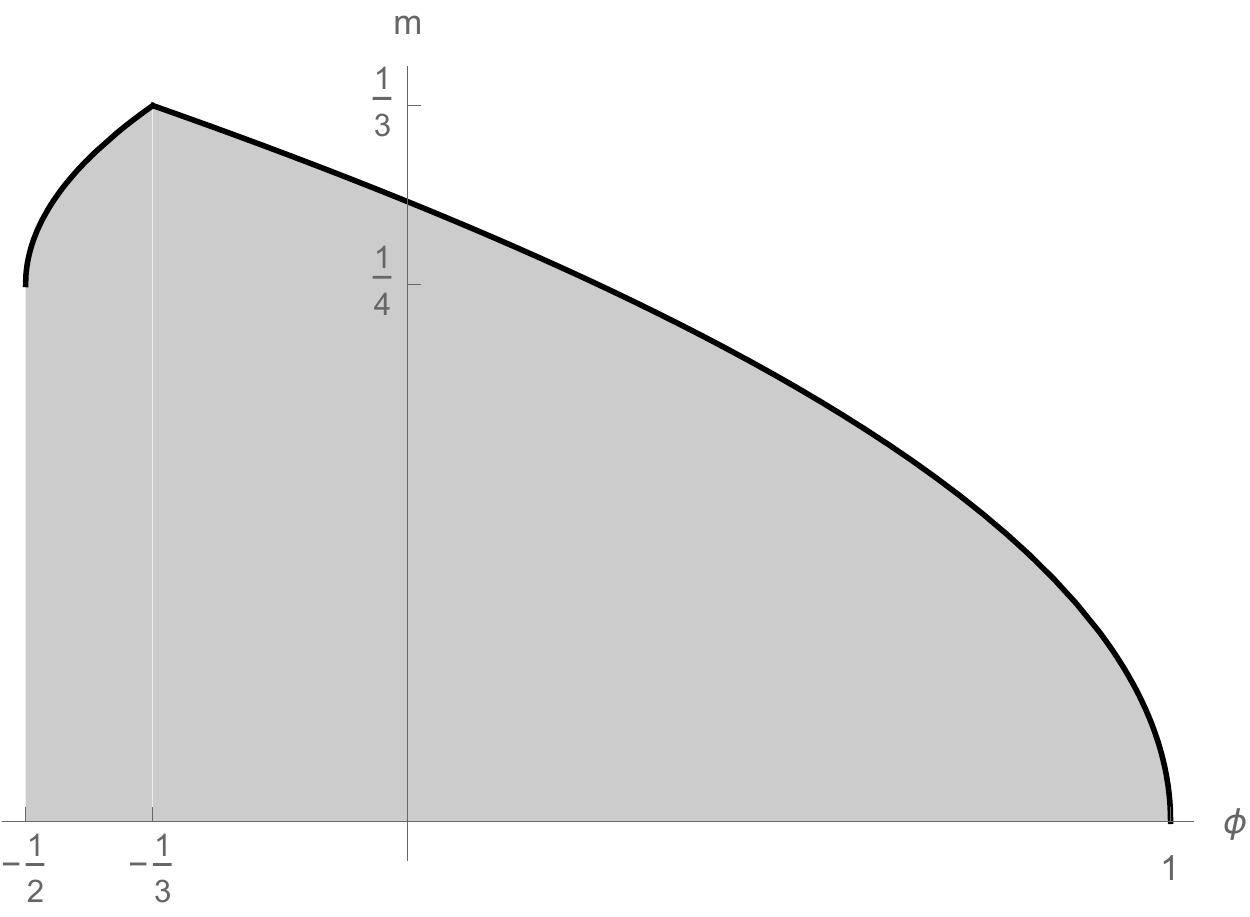}
            \caption{Relations between Spearman's footrule and asymmetry }\label{fig-phi}
\end{figure}

\section{Relations between Gini's gamma and asymmetry }\label{sec:gamma}

To find the relations between Gini's gamma and asymmetry, we compute minimum and maximum from \eqref{eq:meji} for $\kappa=\gamma$. Recall that
$$\gamma(C)=\cQ(M,C)+\cQ(W,C).$$

\begin{lemma}\label{lem:gamma}
  \begin{description}
    \item[(a)] $\displaystyle\min_{(a,b)\in\Delta_m}\gamma(\underline{C}^{(a,b)}_{m})= \gamma(\underline{C}^{(m,1-m)}_{ m})=\left\{ \begin{array}{ll}
                                    4m^2-1; & \text{if } 0 \leqslant m\leqslant\frac14, \vspace{1mm}\\
                                    20m^2-8m; & \text{if } \frac14 \leqslant m\leqslant \frac13,
                \end{array} \right.$
    \item[(b)] $\displaystyle\max_{(a,b)\in\Delta_m}\gamma(\overline{C}^{(a,b)}_{m})= \gamma(\overline{C}^{(m,2m)}_{m})=
        \left\{ \begin{array}{ll}
        1-8m^2;     & \text{if } 0 \leqslant m \leqslant \frac16, \vspace{1mm}\\
        12m-44m^2;  & \text{if } \frac16 \leqslant m \leqslant \frac15, \vspace{1mm}\\
        1+2m-19m^2; & \text{if } \frac15 \leqslant m \leqslant \frac14, \vspace{1mm}\\
        2-6m-3m^2;  & \text{if } \frac14 \leqslant m \leqslant \frac13.
          \end{array} \right.$.
  \end{description}
\end{lemma}

\begin{proof}
{\bf (a): }
By Proposition \ref{prop2} and \eqref{cQ(M,C1)} we have for any $(a,b) \in \Delta_m$
$$\gamma(\underline{C}^{(a,b)}_{m}) = 4m(1-a-b+m)-1+\left\{ \begin{array}{ll}
        0;                        & \text{if } m + \frac12 \leqslant b, \vspace{1mm}\\
        (2m+1-2b)^2;              & \text{if } \frac12(1+m) \leqslant b \leqslant m + \frac12, \vspace{1mm}\\
        m(2+3m-4b);               & \text{if } b \leqslant \frac12(1+m).  \\
                \end{array} \right.$$
For fixed $m$ and $b$  this expression is minimal for $a$ as large as possible, i.e when the point $(a,b)$ lies on the edge $RU$ or $UT$ (see Figure \ref{trikot}). This means that for $b \leqslant \frac12(m+1)$ we can take $a=b-m$ and for $b \geqslant \frac12(m+1)$ we can take $a=1-b$. We get
$$\gamma(\underline{C}^{(a,b)}_{m}) = \left\{ \begin{array}{ll}
        4 m^2-1;                 & \text{if } m + \frac12 \leqslant b, \vspace{1mm}\\
        4 m^2-1+(2m+1-2b)^2;     & \text{if } \frac12(1+m) \leqslant b \leqslant m + \frac12, \vspace{1mm}\\
        m(6-12b+11m)-1;          & \text{if } b \leqslant \frac12(1+m).  \\
                \end{array} \right.$$
This function is decreasing in $b$, since the first part is constant, the third part is linear with linear coefficient
$-12m$, and the second part is quadratic with derivative
$$\frac{d}{db}(4 m^2-1+(2m+1-2b)^2) = - 4(2m+1-2b) \leqslant 0.$$
So, $\gamma(\underline{C}^{(a,b)}_{m})$ takes its minimal value
for $a=m$, $b= 1-m$, i.e. at the vertex $T$, where
$$\gamma(\underline{C}^{(m,1-m)}_{m}) =\left\{ \begin{array}{ll}
        4m^2-1 ;    & \text{if } 0 \leqslant m \leqslant \frac14, \vspace{1mm}\\
        20m^2-8m; & \text{if } \frac14 \leqslant m \leqslant \frac13.
                \end{array} \right.
                $$

\begin{figure}[h]
            \includegraphics[width=5cm]{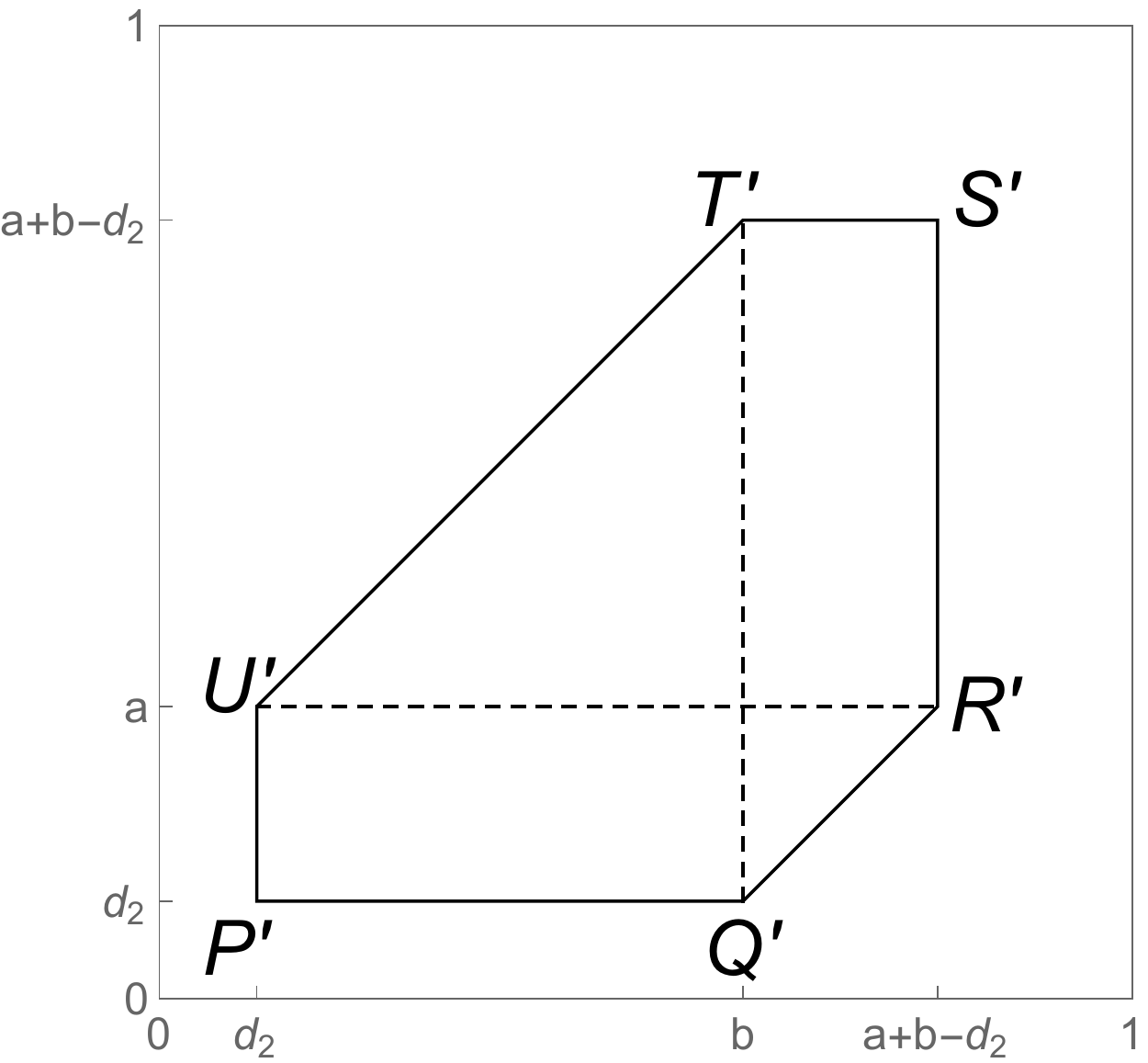}
            \caption{{The hexagon $P'Q'R'S'T'U'$} }\label{sestkotnik2}
\end{figure}

{\bf (b):}
For $(a,b) \in \Delta_m$ in the expression for $\cQ(W, \overline{C}^{(a,b)}_{m})$ four out of nine cases are possible, namely, the first, the second, the fourth, and the fifth. In the third case, when the vertex $T'$ of the hexagon is below the counter-diagonal and $R'$ is above it (see Figure \ref{sestkotnik2}),
the central point $(b, a)$ is above the main diagonal, so $b \leqslant a$. In the remaining four cases we have the central point above the counter-diagonal. Since in the triangle $\Delta_m$, we have $d_2 = a-m$, Proposition \ref{prop3} implies that the expression for  $\cQ(W, \overline{C}^{(a,b)}_{m})$
simplifies to
$$\cQ(W, \overline{C}^{(a,b)}_{m}) = \left\{ \begin{array}{ll}
        0;                               & \text{if } b \leqslant \frac12-m, \vspace{1mm}\\
        -(2b+2m-1)^2;                    & \text{if } \frac12-m \leqslant b \leqslant \frac12-\frac12 m, \vspace{1mm}\\
        -m(4b+3m-2);                     & \text{if } \frac12-\frac12 m \leqslant b \leqslant 1-a-m, \vspace{1mm}\\
        (a-1)^2+(b-1)^2+2(a-m)(b+m)-1;  & \text{if } 1-a-m \leqslant b.
                \end{array} \right.$$
Since $\cQ(M, \overline{C}^{(a,b)}_{m})= 1-4m(b-a+m)$ we have
$$\gamma(\overline{C}^{(a,b)}_{m}) = 1-4m(b-a+m) + \cQ(W, \overline{C}^{(a,b)}_{m}).$$
For fixed $m$ and $b$  this expression is maximal for $a$ as large as possible, i.e when the point $(a,b)$ lies on the edge $RU$ or $UT$ (see Figure \ref{trikot}). This means that for $b \leqslant \frac12(m+1)$ we can take $a=b-m$ and for $b \geqslant \frac12(m+1)$ we can take $a=1-b$. Then we get
\begin{equation}\label{gam1}
\gamma(\overline{C}^{(a,b)}_{m}) = \left\{ \begin{array}{ll}
        1-8m^2;                  & \text{if } b \leqslant \frac12-m, \vspace{1mm}\\
        1-8m^2-(2b+2m-1)^2;      & \text{if } \frac12-m \leqslant b \leqslant \frac12-\frac12 m, \vspace{1mm}\\
        1-8m^2-m(4b+3m-2);       & \text{if } \frac12-\frac12 m \leqslant b \leqslant \frac12, \vspace{1mm}\\
        1-8m^2+(2b+m-1)(2b-3m-1);  & \text{if } \frac12 \leqslant b \leqslant \frac12+\frac12 m,
        \vspace{1mm}\\ 1-6m(2b+m-1);            & \text{if } \frac12+\frac12 m \leqslant b.
                \end{array} \right.
\end{equation}
This function is decreasing in $b$, since the first part is constant, the third and the fifth are linear with linear coefficients
$-4m$ and $-12m$ respectively, and the second and the fourth are quadratic with derivatives
$\frac{d}{db}(1-8m^2-(2b+2m-1)^2) = -4(2b+2m-1) \leqslant 0$ and $\frac{d}{db}(1-8m^2+(2b+m-1)(2b-3m-1)) = 4(2b-m-1) \leqslant 0$.
So, $\gamma(\overline{C}^{(a,b)}_{m})$ takes its maximal value for $a=m, b= 2m$, i.e. at vertex $R$. Since $m\leqslant\frac13$, the last case in \eqref{gam1} does not occur when $b = 2m$.  Therefore, we obtain
$$\gamma(\overline{C}^{(m,2m)}_{m}) = \left\{ \begin{array}{ll}
        1-8m^2;     & \text{if } 0 \leqslant m \leqslant \frac16, \vspace{1mm}\\
        12m-44m^2;  & \text{if } \frac16 \leqslant m \leqslant \frac15, \vspace{1mm}\\
        1+2m-19m^2; & \text{if } \frac15 \leqslant m \leqslant \frac14, \vspace{1mm}\\
        2-6m-3m^2;  & \text{if } \frac14 \leqslant m \leqslant \frac13.
          \end{array} \right.$$
\end{proof}

The main theorem and its corollary now follow. Figure \ref{fig-gamma} shows the relations between Gini's gamma and asymmetry.

\begin{theorem}\label{th:gamma}
Let $C\in \Cmt$ be any copula with $\mu_\infty(C)=m$. Then $\gamma(C)\in[g(m),h(m)]$, where
$$g(m)=  \left\{ \begin{array}{ll}
                                    4m^2-1; & \text{if } 0 \leqslant m\leqslant\frac14, \vspace{1mm}\\
                                    20m^2-8m; & \text{if } \frac14 \leqslant m\leqslant \frac13,
                \end{array} \right.$$
and
$$h(m)=\left\{
    \begin{array}{ll}
        1-8m^2;     & \text{if } 0 \leqslant m \leqslant \frac16, \vspace{1mm}\\
        12m-44m^2;  & \text{if } \frac16 \leqslant m \leqslant \frac15, \vspace{1mm}\\
        1+2m-19m^2; & \text{if } \frac15 \leqslant m \leqslant \frac14, \vspace{1mm}\\
        2-6m-3m^2;  & \text{if } \frac14 \leqslant m \leqslant \frac13.
    \end{array}
         \right.$$
and the bounds are attained. In particular, if copula $C$ is symmetric, then $\gamma(C)$ can take any value in $[-1,1]$. If copula $C$ takes the maximal possible asymmetry $\mu_\infty(C)=\frac13$, then $\gamma(C) \in [-\frac49,-\frac13]$.
\end{theorem}

\begin{corollary}\label{cor:gamma}
  If $\gamma(C)=\gamma$, then
  \[
    0\leqslant\mu_\infty(C)\leqslant\left\{
           \begin{array}{ll}
             \displaystyle\frac{\sqrt{\gamma+1}}{2}; & \text{if } -1\leqslant \gamma \leqslant -\frac34, \vspace{1mm}\\
             \displaystyle\frac{\sqrt{5\gamma+4}+2}{10}; & \text{if } -\frac34 \leqslant \gamma \leqslant -\frac49, \vspace{1mm}\\
             \displaystyle\frac{1}{3}; & \text{if } -\frac49 \leqslant \gamma \leqslant -\frac13, \vspace{1mm}\\
             \displaystyle\frac{\sqrt{15-3\gamma}-3}{3}; & \text{if } -\frac13 \leqslant \gamma \leqslant \frac{5}{16}, \vspace{1mm}\\
             \displaystyle\frac{\sqrt{20-19\gamma}+1}{19}; & \text{if } \frac{5}{16} \leqslant \gamma \leqslant \frac{16}{25}, \vspace{1mm}\\
             \displaystyle\frac{\sqrt{9-11\gamma}+3}{22}; & \text{if } \frac{16}{25} \leqslant \gamma \leqslant \frac79, \vspace{1mm}\\
             \displaystyle\frac{\sqrt{2-2\gamma}}{4}; & \text{if } \frac79 \leqslant \gamma \leqslant 1.
     \end{array}\right.
  \]
and the bounds are attained.
\end{corollary}

\begin{figure}[h]
            \includegraphics[width=4cm]{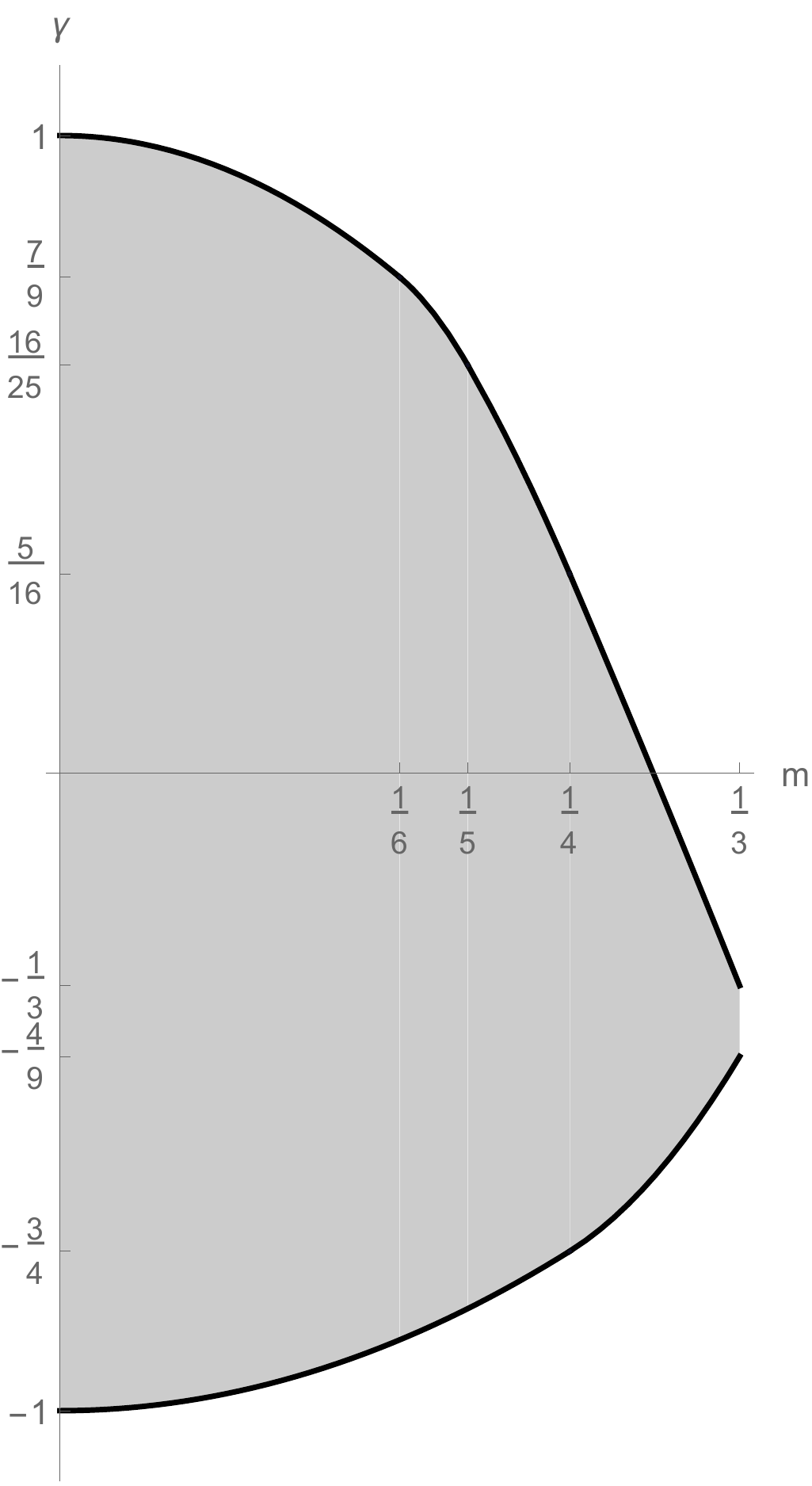} \hfil \includegraphics[width=6cm]{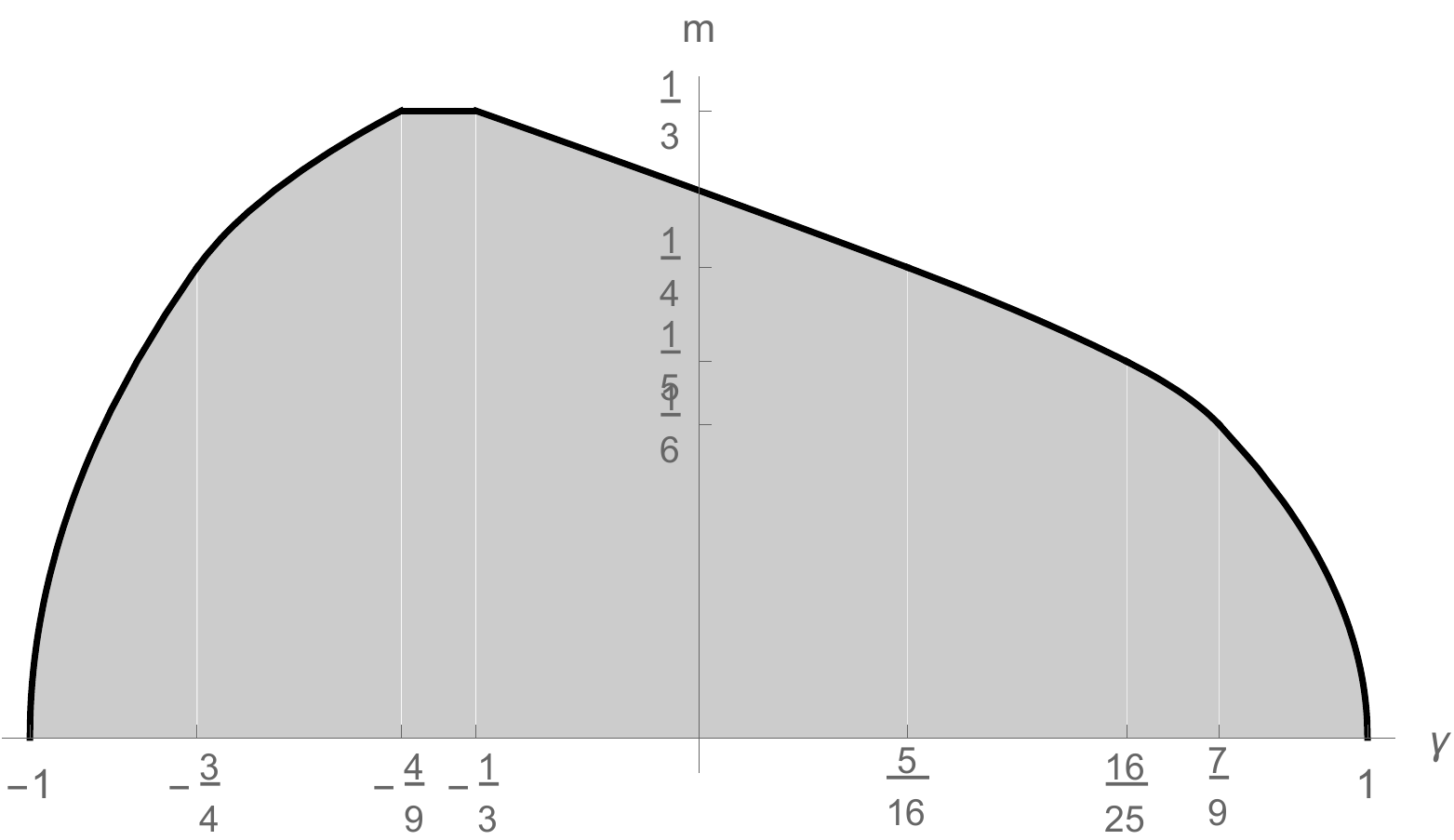}
            \caption{Relations between Gini's gamma and asymmetry }\label{fig-gamma}
\end{figure}

\section{Relations between Blomqvist's beta and asymmetry }\label{sec:beta}

Blomqvist's beta $\beta(C)$ of copula $C$ is defined in terms of $C(\frac12,\frac12)$ (see (\ref{beta})). Observe that for $a \leqslant b$
$$\underline{C}^{(a,b)}_{c}\textstyle(\frac12,\frac12)=\left\{
    \begin{array}{ll}
             c;           & a,b\leqslant\frac12 \text{ or } \frac12\leqslant a,b\vspace{1mm}\\
             c+a-\frac12; & a\leqslant\frac12, \frac12\leqslant b, a+b\geqslant 1, c+a\geqslant\frac12, \vspace{1mm}\\
             c-b+\frac12; & a\leqslant\frac12, \frac12\leqslant b, a+b\leqslant 1, c-b\geqslant\frac12, \vspace{1mm}\\
             0;           & \text{otherwise}.
    \end{array}
         \right.
$$
Then it is straightforward to observe that
$$\min_{(a,b)\in\Delta RST} \beta(\underline{C}^{(a,b)}_{m})=\beta(\underline{C}^{(m,1-m)}_{m})=\left\{
    \begin{array}{ll}
             -1,& 0\leqslant m\leqslant\frac14,\vspace{1mm}\\
             8m-3,& \frac14\leqslant m\leqslant\frac13.
    \end{array}
         \right.$$
This gives the lower bound of $\beta$ for all copulas with asymmetry $m$. It is attained at vertex $T$. Intuitively, to minimize $\beta$, the bump of $\underline{C}^{(a,b)}_{m}$ needs to be as far away from the point $(\frac12, \frac12)$ as possible.

Let us now consider the upper bound. For $a \leqslant b$ we have
$$\overline{C}^{(a,b)}_{c}\textstyle(\frac12,\frac12)=\left\{
    \begin{array}{ll}
             a-c; & a-c \leqslant \frac12 \leqslant a,\vspace{1mm}\\
             \frac12-c; & a \leqslant \frac12 \leqslant b, \vspace{1mm}\\
             1-b-c; & b \leqslant \frac12 \leqslant b+c, \vspace{1mm}\\
             \frac12,& \rm{otherwise}.
    \end{array}
         \right.
$$
Then it is straightforward to observe that
$$\max_{(a,b)\in\Delta RST} \beta(\overline{C}^{(a,b)}_{m})=\beta(\overline{C}^{(m,2m)}_{m})=\left\{
    \begin{array}{ll}
             1,& 0\leqslant m\leqslant\frac16,\vspace{1mm}\\
             3-12m,& \frac16\leqslant m\leqslant\frac14,\vspace{1mm}\\
             1-4m,&\frac14\leqslant m\leqslant\frac13.
    \end{array}
         \right.$$
The upper bound is attained at vertex $R$.

We have proved the following theorem and its corollary. Figure \ref{fig-beta} shows the relations between Blomqvist's beta and asymmetry.

\begin{theorem}\label{th:beta}
Let $C\in \Cmt$ be any copula with $\mu_\infty(C)=m$. Then $\beta(C)\in[g(m),h(m)]$, where
$$g(m)=  \left\{
    \begin{array}{ll}
             -1,& 0\leqslant m\leqslant\frac14,\vspace{1mm}\\
             8m-3,& \frac14\leqslant m\leqslant\frac13.
    \end{array}
         \right.$$
and
$$h(m)=\left\{
    \begin{array}{ll}
             1,& 0\leqslant m\leqslant\frac16,\vspace{1mm}\\
             3-12m,& \frac16\leqslant m\leqslant\frac14,\vspace{1mm}\\
             1-4m,&\frac14\leqslant m\leqslant\frac13.
    \end{array}
         \right.$$
and the bounds are attained. In particular, if copula $C$ is symmetric, then $\beta(C)$ can take any value in $[-1,1]$. If copula $C$ takes the maximal possible asymmetry $\mu_\infty(C)=\frac13$, then $\beta(C) = -\frac13$.
\end{theorem}

\begin{figure}[h]
            \includegraphics[width=4cm]{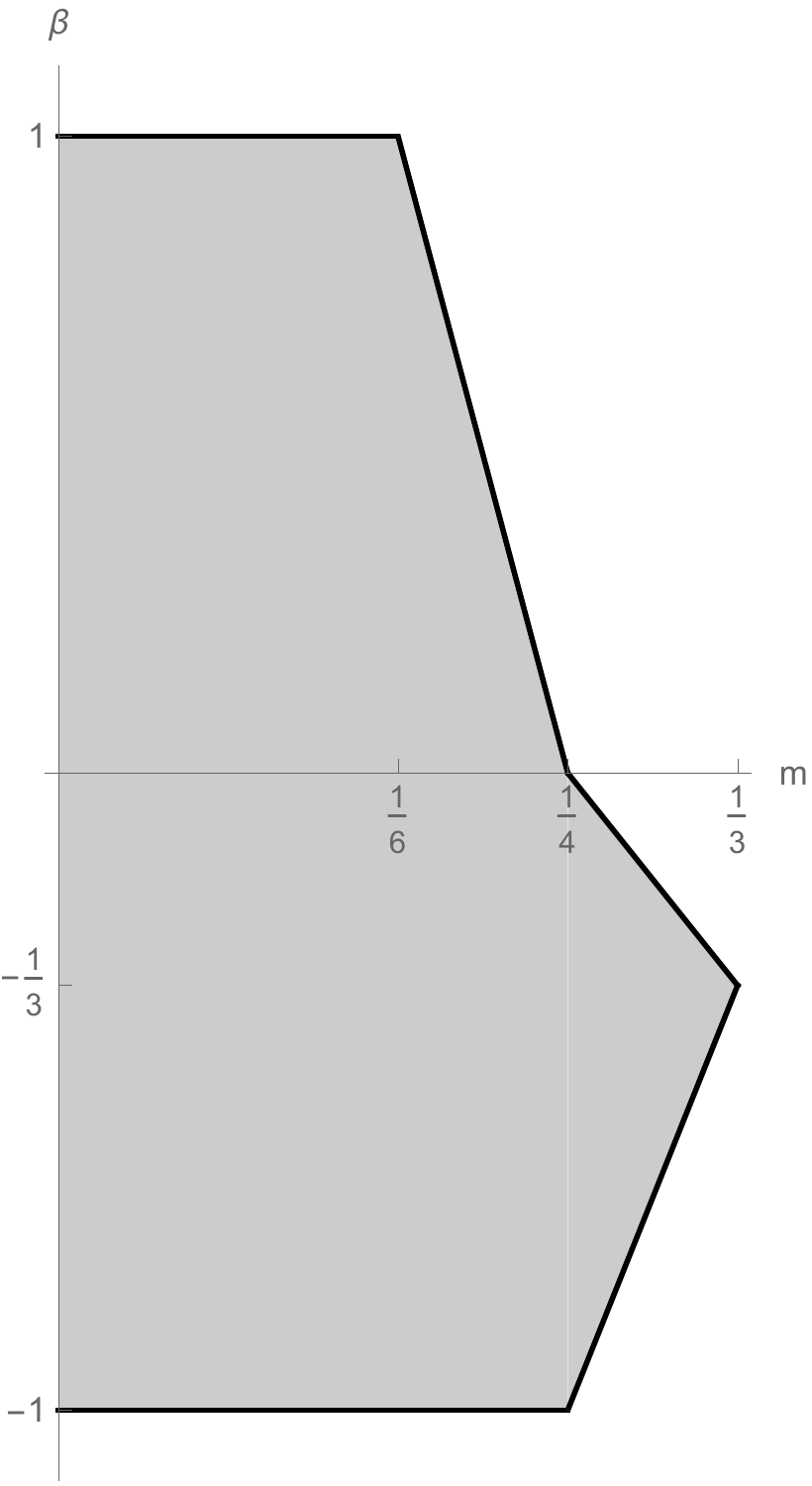} \hfil \includegraphics[width=6cm]{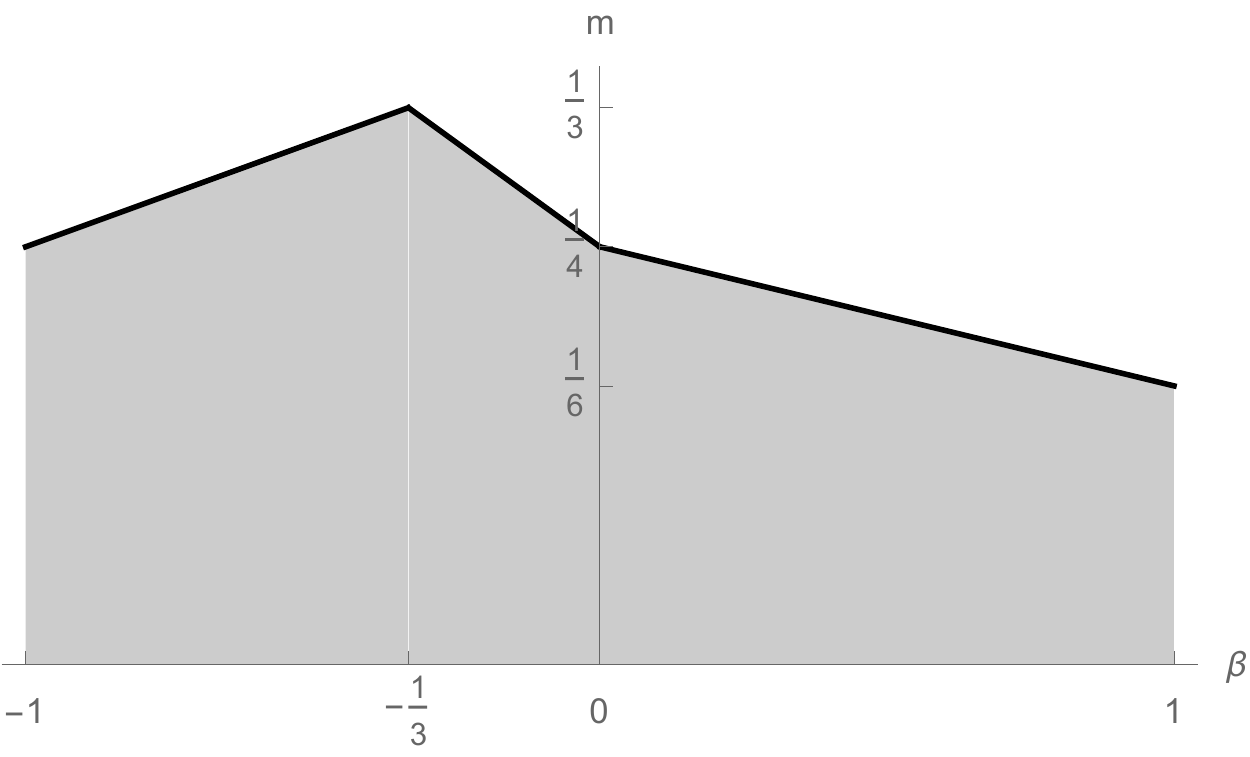}
            \caption{Relations between Blomqvist's beta and asymmetry }\label{fig-beta}
\end{figure}

\begin{corollary}\label{cor:beta}
  If $\beta(C)=\beta$, then
  \[
    0\leqslant\mu_\infty(C)\leqslant\left\{
           \begin{array}{ll}
             \displaystyle\frac{3+\beta}{8}; & \text{if } -1\leqslant \beta\leqslant-\frac13, \vspace{1mm}\\
             \displaystyle\frac{1-\beta}{4}; & \text{if } -\frac13\leqslant\beta\leqslant 0, \vspace{1mm}\\
             \displaystyle\frac{3-\beta}{12}; & \text{if } 0\leqslant\beta\leqslant 1.
           \end{array}
         \right.
  \]
and the bounds are attained.
\end{corollary}

\section{Conclusion}

Here is a somewhat surprising fact that emerges from our main results.

\begin{corollary}
Let $m\in [0,\frac13]$.
For any measure of concordance $\kappa$ from the set $\{\rho, \tau, \phi, \gamma, \beta\}$ the maximal value of $\kappa(\overline{C}^{(a,b)}_{m})$ over the triangle $\Delta_m$ is attained at the same point, namely $R(m,2m)$.

For any measure of concordance $\kappa$ from the set $\{\rho, {\tau,} \phi, \gamma, \beta\}$ the minimal value of $\kappa(\underline{C}^{(a,b)}_{m})$ over the triangle $\Delta_m$ is attained at the same point, namely $T(m,1-m)$.

{Furthermore, the following table summarizes the minimizing and maximizing line segments (or points) in the triangle $RUT$ (see Figure \ref{trikot}):
\begin{table}[h]
\begin{tabular}{c|ccccc|}
\cline{2-6}
                      &  $\rho$ & $\tau$ &  $\phi$ & $\gamma$ &  $\beta$ \\ \hline
\multicolumn{1}{|c|}{ $\min$}  & $UT$ & $UT$ &  $T$ & $T$ & $T$ \\ \hline
\multicolumn{1}{|c|}{ $\max$}  &  $RU$ &  $RU$ &  $RU$ &  $R$  &  $R$  \\ \hline
\end{tabular}
\end{table}
}
\end{corollary}


Observe that the minimizing and maximizing segments in the above table are reflected about the line $b=\frac{m+1}{2}$, the symmetry line of the $\Delta_m$ triangle.

We also want to point out another fact for measures of concordance from the set  $\{\rho, \phi, \gamma, \beta\}$~: A high value of asymmetry, i.e., a value of $\mu_{\infty}(C)$ from an interval $(\delta_{\kappa},\frac13]$, implies negative value of $\kappa(C)$ for $\kappa$ in this set. Here $\delta_{\kappa}\ge\frac14$ is a number depending on $\kappa$. In short, high asymmetry implies negative measure of concordance for all measures of concordance we consider except for Kendall's tau. {However, it is not surprising that Kendall's tau behaves differently compared to the other measures of concordance. Namely, $\rho, \phi, \gamma$ and $\beta$ are linear with respect to convex combinations, while $\tau$ is quadratic \cite{EdTa}. }

\medskip

This paper presents the relations between asymmetry and measures of concordance which are nicely seen on our Figures \ref{fig-rho}, \ref{fig-tau}, \ref{fig-phi}, \ref{fig-gamma}, and \ref{fig-beta}.
This study may serve as an encouragement to apply the method proposed here also to investigate other relations between these and other properties and measures of dependence structure of copulas in order to help practitioners to choose their copula models in accordance with their data.
\vskip 12pt

\noindent\textbf{Acknowledgement.} The authors are thankful to Fabrizio Durante for proposing the study of relation between concordance and asymmetry. {Our special thanks are due to the referees. The suggestions of one of them helped us to greatly improve the paper.} The figures in the paper were drawn using the Mathematica software \cite{math}.


\begin{thebibliography}{00}




{\bibitem{BeDoUF}
J. Behboodian,  A. Dolati,  M. \'Ubeda-Flores. {\it Measures of association based on average
quadrant dependence.} J. Probab. Stat. Sci. \textbf{3}  (2005), No. 1, 161--173.

\bibitem{BeDoUF2}
J. Behboodian,  A. Dolati,  M. \'Ubeda-Flores. {\it A multivariate version of Gini's rank association coefficient.} Statistical Papers \textbf{48}  (2007), 295--304.}

\bibitem{BBMNU14} G. Beliakov, B. DeBaets, H. DeMeyer, R. B. Nelsen, M. \`Ubeda-Flores. {\it Best-possible bounds on the set of copulas with given degree of non-exchangeability}. J. Math. Anal. Appl. \textbf{417} (2014), 451--468



{\bibitem{DeBDeMJw}
B. De Baets, H. De Meyer, T. Jwaid. {\it On the degree of asymmetry of a
quasi-copula with respect to a curve.} Fuzzy Sets and Systems \textbf{354} (2019),
84--103.}


{\bibitem{DuFeSaSe}
F. Durante, J. Fern\'andez-S\'anchez, C. Sempi. {\it A topological
proof of Sklar's theorem.} Appl. Math. Lett. \textbf{26} (2013), No. 9, 945--948.





\bibitem{DuFu}
F. Durante, S. Fuchs. {\it Reflection invariant copulas.} Fuzzy Sets and Systems \textbf{354} (2019), 63--75.}

\bibitem{DKSU-F}
F. Durante, E. P. Klement, C. Sempi,  M. \'Ubeda-Flores.
{\it Measures of non-exchangeability for bivariate random vectors}.
Statist. Papers {\bf 51} (2010), no. 3, 687--699.




\bibitem{DM}
F. Durante, R. Mesiar.
{\it $L_{\infty}$-measure of non-exchangeability for bivariate extreme value and Archimax copulas.}
J. Math. Anal. Appl. {\bf 369} (2010), no. 2, 610--615.

\bibitem{DP2}
F. Durante, P. L. Papini.
{\it Componentwise concave copulas and their asymmetry.}
Kybernetika {\bf 45} (2009), no. 6, 1003--1011.

\bibitem{DP1}
F. Durante, P. L. Papini.
{\it Non-exchangeability of negatively dependent random variables.}
Metrika {\bf 71} (2010), no. 2, 139--149.

\bibitem{DuSe} F.~Durante, C.~Sempi. {\sl Principles of Copula Theory}. CRC/Chapman \& Hall, Boca Raton (2015).

{\bibitem{EdMiTa}
H. H. Edwards,  P. Mikusi\'nski, M. D. Taylor. {\it Measures of concordance determined by
$D_4$-invariant copulas.} Int. J. Math. \& Math. Sci. \textbf{70} (2004), 3867--3875.


\bibitem{EdMiTa2}
H. H. Edwards,  P. Mikusi\'nski,  M. D. Taylor. {\it Measures of concordance determined by
$D_4$-invariant measures on $(0, 1)^2$.} Proc. Amer. Math. Soc. \textbf{133} (2005), No. 5, 1505--1513.

\bibitem{EdTa}
H. H. Edwards, M. D. Taylor. {\it Characterizations of degree one bivariate measures of
concordance.} J. Multivariate Anal. \textbf{100}  (2009), 1777--1791.

\bibitem{FeSaUbFl}
J. Fern\'andez-S\'anchez, M. \'Ubeda-Flores. {\it On degrees of asymmetry
of a copula with respect to a track.} Fuzzy Sets and Systems \textbf{354} (2019),
104--115.}

\bibitem{FrNe} G.\ A.\ Fredricks, R.\ B.\ Nelsen. \textit{On the relationship between Spearman's rho and Kendall's tau for pairs of continuous random variables}. Journal of Statistical Planning and Inference \textbf{137} (2007), 2143--2150.


{\bibitem{FuMcCSch}
S. Fuchs, Y. McCord,  K. D. Schmidt. {\it Characterizations of copulas attaining the bounds
of multivariate Kendall's tau.} J. Optim. Theory Appl. \textbf{178}  (2018), No. 2, 424--438.}

{\bibitem{FuSch}
S. Fuchs, K. D. Schmidt. {\it Bivariate copulas: Transformations, asymmetry and measures of
concordance.} Kybernetika \textbf{50} (2014), 109--125.}

\bibitem{GeNe} C.\ Genest, J.\ Ne\v{s}lehov\'a. \textit{Assessing and Modeling Asymmetry in Bivariate Continuous Data}. In: P.\ Jaworski, F.\ Durante, W.K.\ H\"ardle, (eds.), {C}opulae in {M}athematical and {Q}uantitative {F}inance, Lecture Notes in Statistics, Springer Berlin Heidelberg, (2013), 91--114.

{\bibitem{GeNeGh}
C.\ Genest, J.\ Ne\v{s}lehov\'a, N. B. Ghorbal. {\it Spearman's footrule and Gini's gamma: a review with complements.} J. Nonparametric Statistics \textbf{22} (2010), No. 8, 937--954.



\bibitem{Joe97}
H. Joe. {\sl Multivariate Models and Multivariate Dependence Concepts.} Chapman \& Hall, London, 1997.

\bibitem{Joe}
H.\ Joe. {\sl Dependence Modeling with Copulas.} Chapman \& Hall/CRC, London, 2014.}



\bibitem{KlMe}
E. P. Klement and R. Mesiar.
{\it How non-symmetric can a copula be?} Comment. Math. Univ. Carolin. {\bf 47} (2006), no. 1, 141--148.



\bibitem{KoBuKoMoOm} D.\ Kokol Bukov\v{s}ek, T.\ Ko\v{s}ir, B.\ Moj\v{s}kerc, M. Omladi\v{c}. {\it Non-exchangeability of copulas arising from shock models}, J.\ of Comp.\ and Appl.\ Math., \textbf{358} (2019), 61--83. (See also {Erratum} published in J.\ of Comp.\ and Appl.\ Math., \textbf{365} (2020), https://doi.org/10.1016/j.cam.2019.112419 and corrected version at https://arxiv.org/abs/1808.09698v4).


\bibitem{Krus}  W. H. Kruskal. {\it Ordinal measures of association}. J. Amer. Stat. Soc., \textbf{53} (1958), 814--861.

\bibitem{Lieb}
E. Liebscher. {\it Copula-based dependence measures.} Depend. Model. \textbf{2} (2014), 49--64.












\bibitem{Nels} R.\ B.\ Nelsen. {\sl An introduction to copulas}. 2nd edition, Springer-Verlag, New York (2006).

\bibitem{N}
R. B. Nelsen. {\it Extremes of nonexchangeability}. Statist. Papers {\bf 48} (2007), No. 2, 329--336.

\bibitem{NQRU} R. B. Nelsen, J. J. Quesada-Molina, J. A. Rodr\'{\i}guez-Lallena, M. \'{U}beda-Flores. {\it Bounds on bivariate distribution functions with given margins and measures of association}. Commun. Statist. Theory Meth., \textbf{30} (2001), No. 6, 1155--1162.

{\bibitem{NQRU2} R. B. Nelsen, J. J. Quesada-Molina, J. A. Rodr\'{\i}guez-Lallena, M. \'{U}beda-Flores. {\it Distribution functions of copulas: a class of bivariate probability integral tranforms}. Statist. \& Probab. Letters \textbf{54} (2001), 277--282.}

\bibitem{NeUbF} R. B. Nelsen, M. \'{U}beda-Flores. {\it A comparison of bounds on sets of joint distribution functions derived from various measures of association.} Commun. Statist. Theory Meth., \textbf{33} (2004), No. 10, 2299--2305.

{\bibitem{NeUbF2}
R. B. Nelsen, M. \'Ubeda-Flores. {\it The lattice-theoretic structure of
sets of bivariate copulas and quasi-copulas.} C. R. Math. Acad. Sci. Paris \textbf{341}
(2005), No. 9, 583--586.}








{\bibitem{Scar}
M. Scarsini. {\it On measures of concordance.} Stochastica \textbf{8} (1984), No. 3, 201--218.}

\bibitem{Skla} A.\ Sklar. {\it Fonctions de r\'{e}partition \`{a} $n$ dimensions et leurs marges}. Publ.\ Inst.\ Stat.\ Univ.\ Paris \textbf{8} (1959), 229--231.

{\bibitem{Tayl}
M. D. Taylor. {\it Multivariate measures of concordance for copulas and their marginals.} Depend.
Model. \textbf{4} (2016), 224--236.

\bibitem{UbFl}
M. \'Ubeda-Flores. {\it Multivatiate versions of Blomqvist's beta and Spearman's footrule.} Ann. Inst. Statist. Math. \textbf{57} (2005), No. 4, 781--788.

\bibitem{math} Wolfram Research, Inc. Mathematica, Version {\bf 11}, Champaign, IL, 2017.
}

\end{thebibliography}
\end{document}